\documentclass[reqno]{amsart} 
\usepackage{amssymb}
\usepackage[utf8]{inputenc}
\usepackage[sorting=nyt]{biblatex}
\usepackage[hidelinks]{hyperref}
\usepackage[T1]{fontenc}

\usepackage{kotex} \usepackage[dvipsnames]{xcolor}

\usepackage{colonequals}
\usepackage{theoremref}
\usepackage{bbm}
\usepackage{enumitem}

\newtheorem{thm}{Theorem}[section] 
\newtheorem{prop}[thm]{Proposition} 
\newtheorem{lem}[thm]{Lemma}

\theoremstyle{remark}
\newtheorem*{rmk}{Remark}

\theoremstyle{definition}
\newtheorem*{defi}{Definition}

\newtheorem*{ack}{Acknowledgments}

\title[cutoff phenomenon on complete multipartite graphs]{Cutoff phenomenon of the Glauber dynamics for the Ising model on complete multipartite graphs in the high temperature regime}

\author{Heejune Kim}
\address{School of Mathematics, University of Minnesota.
}
\email{kim01154@umn.edu}

\subjclass[2020]{60J10, 60K35, 82C20}
\keywords{Markov chains, Ising model, Mixing time, Cutoff, Coupling, Glauber dynamics, Heat-bath dynamics, Mean-ﬁeld model}

\begin{document}

\begin{abstract}
    In this paper, the Glauber dynamics for the Ising model on the complete multipartite graph $K_{np_1,\dots,np_m}$ is investigated where $0<p_i<1$ is the proportion of the vertices in the $i$th component.
    We show that the dynamics exhibits the cutoff phenomena at $t_n \colonequals \frac{1}{2(1-\beta/\beta_{cr})} n\ln n $ with window size $O(n)$ in the high temperature regime $\beta< \beta_{cr}$ where $\beta_{cr}$ is a constant only depending on $p_1,\dots,p_m$.
    Exponentially slow mixing is shown in the low temperature regime $\beta>\beta_{cr}$.
\end{abstract}

\maketitle

\section{Introduction and preliminaries} 
Informally, the \emph{cutoff phenomenon} is an abrupt transition of a Markov chain to its equilibrium when the system under consideration is sufficiently large (see Section \ref{subsection 1.2} for a rigorous definition).
To the author's knowledge, the first rapid mixing result appeared in \cite{diaconis} on the symmetric group while considering random transpositions.
Shortly afterward, \textcite{aldous} showed that the top-in-at-random card-shuffle precisely exhibits a cutoff phenomenon, initiating the whole industry of the cutoff phenomenon.

As pointed out in \cite{Lubetzky_2012}, only a few examples of cutoff were known regarding the Glauber dynamics of the Ising model (see Section \ref{subsection1.1} for formal definitions), such as that of \cites{ding}{levin} on complete graphs and of \cites{Lubetzky_2012}{lubetzky2012cutoff}{lubetzky2014} on lattices.
Recent researches have mainly focused on lattices. 
A breakthrough paper by \textcite{Lubetzky_2012} showed cutoff with a continuous-time window $O(\ln\ln n)$ for this longstanding problem.
An improvement on the window size to optimal $O(1)$ was made by the same authors in \cite{lubetzky2014} with the information percolation framework.
By the same technique, the authors illustrated the existence of cutoff in high enough temperatures for the Ising model of any sequence of graphs with a bounded degree in \cite{lubetzky2017}. 
Mean-field Potts model on complete graphs was comprehensively explored in \cite{potts}, again verifying the cutoff phenomenon in high temperatures.
For the bipartite Potts model, \textcite{Yevgeniy} proved the cutoff phenomena in the high temperatures using their aggregate path coupling method.

The purpose of this paper is to investigate the Glauber dynamics for the Ising model on complete multipartite graphs. 
(Exact definitions are given in the rest of the introduction.)
Indeed, we identify the critical temperature and establish cutoff in the high temperature regime.
On the other hand, exponentially slow mixing is established in the low temperature regime.
The significance of our setting is that complete multipartite graphs have an intermediate geometry between the complete graphs which have no geometry at all (e.g. \cite{levin}), and lattices which have a strong geometry (e.g. \cite{Lubetzky_2012}).
Thus, our result serves as a midway example between those two extreme cases.
The method of proof hinges on generalizations of the tools in \cite{levin}, notably the two-coordinate chain thereof.

Due to the nature of complete multipartite graphs, our model can be considered as a block spin Ising model with no interaction inside each block.
Such mean-field block models naturally occur in statistical physics when modelling metamagnets (see \cite{KINCAID197557}) and in studies on social interactions (see, e.g., \cite{Gallo}).
A recent paper by \textcite{Knöpfel2020} contains an excellent introduction to this line of work.

When it comes to cutoff phenomenon on finite graphs, it is easy to convert the discrete-time results to that of the continuous-time and vice versa.
Hence, we only consider discrete-time chains.

\subsection{Notations}
Boldface letters are used to denote vectors or matrices.
Inequalities between vectors and matrices are defined element-wise.
The dependence of any quantities on the number of vertices $n$ is understood throughout the paper.
Some important quantities not depending on $n$ will be explicitly mentioned.
We will write $\mathbf e_j$ to be the $j$th vector in the standard basis of $\mathbb R^m$.
The lower case $t$ will always denote time.
Let $\circ$ denote the Hadamard product between matrices.
More precisely, $B\circ C = (B_{ij} C_{ij})$ whenever $B=(B_{ij})$ and $C=(C_{ij})$ are matrices with the same dimensions.

\subsection{Ising model and Glauber dynamics}\label{subsection1.1}
Let $G=(V,E)$ be a finite graph with the vertex set $V$ and the edge set $E$.
Elements of $\Omega \colonequals \{\pm 1\}^V$ are called \emph{configurations}. 
In the absence of external fields, the \emph{Ising model} on $G$ is a distribution $\mu$ called the \emph{Gibbs distribution} on $\Omega$ given by \[\mu (\sigma) \colonequals \frac{e^{-\beta H(\sigma)}}{Z(\beta) }\] where $\sigma \in \Omega$, $\beta \geq 0$, $H(\sigma)=- \sum_{ij \in E}h_{ij}\sigma(i) \sigma(j)$, and $Z(\beta)$ is a normalizing factor.
Assuming an isotropic interaction strength between the vertices, we set $h_{ij}=1/|V|$.
The physical interpretation of $H(\sigma)$ is the energy of the whole spin system with the configuration $\sigma$.
We call each $\sigma(v)$ the \emph{spin} at \emph{site} $v$.

The \emph{Glauber dynamics} for the Ising model is a reversible Markov chain with respect to the Gibbs distribution satisfying the following rule. 
At each time, choose a site uniformly at random in $V$ and update the spin at the chosen site according to $\mu$ conditioned on the set of configurations having the same spins at all the sites except the chosen one.
The Glauber dynamics for the Gibbs distribution $\mu$ is irreducible, aperiodic, and reversible with $\mu$ as its unique stationary distribution.
For the Ising model, it is easy to see that the probability of updating to $\pm 1$ at the chosen site $v$ is $r_{\pm}(S)$ where \begin{gather}
    r_{\pm}(x) \colonequals \frac{e^{\pm \beta x}}{ e^{\beta x}+e^{-\beta x}} = \frac{1 \pm \tanh (\beta x)}{2}; \quad x\in\mathbb R  \label{r+}
\end{gather}  and $S=\sum_{vv' \in E} \sigma(v')/|V|$ is the mean-field at $v$.

\subsection{Markov chain mixing and cutoff phenomenon} \label{subsection 1.2}
The \emph{total variation distance} between two probability measures $\nu_1$ and $\nu_2$ on $\Omega$ is defined by \[\|\nu_1 -\nu_2\|_{TV} \colonequals \sup_{A\subseteq \Omega}|\nu_1(A)-\nu_2(A)|= \frac{1}{2}\sum_{x\in \Omega}|\nu_1(x)-\nu_2(x)|.\]
The total variation distance is half of the $L^1$-distance between the probability measures.

Let $({\sigma_{t}})$ be the Markov chain of the Glauber dynamics for the Ising model.
Define the worst-case total variation distance of the chains to the stationary distribution $\mu$ at time $t$ by \[d(t) \colonequals \max_{\sigma \in \Omega} \|\mathbb P_{\sigma }({\sigma_{t}} \in \cdot) -\mu\|_{TV}\] where here and thereafter $\mathbb P_{\sigma}$ denotes the probability given $\sigma_0 =\sigma$.
The \emph{mixing time} is defined by \[t_{\mathrm{mix}}(\varepsilon) \colonequals \min\{t : d(t) \leq \varepsilon\}; \quad \varepsilon \in (0,1).\]

We say a sequence of Markov chains with corresponding mixing times $t_{\mathrm{mix}}^{(n)}(\varepsilon)$ exhibit a \emph{cutoff phenomenon} if for every $0<\varepsilon<1/2$, \[\lim_{n \to \infty} \frac{t_{\mathrm {mix}}^{(n)}(\varepsilon)}{t_{\mathrm mix}^{(n)}(1-\varepsilon)}=1.\]
Furthermore, we say that the cutoff occurs at $t_{\mathrm {mix}}^{(n)}$ with \emph{window size} $O(w_n)$ if $w_n= o(t_{\mathrm {mix}}^{(n)})$ and \begin{gather*}
    \lim_{\gamma \to \infty }\liminf_{n \to \infty }d_n(t_{\mathrm {mix}}^{(n)}-\gamma w_n)=1,\quad \lim_{\gamma \to \infty }\limsup_{n \to \infty }d_n(t_{\mathrm {mix}}^{(n)}+\gamma w_n)=0.
\end{gather*}

\subsection{Magnetization chain on complete multipartite graphs}
Now, we are in a place to consider a complete $m$-partite graph, a graph whose vertices are partitioned into $m$ different independent sets, and every pair of vertices from different independent sets is connected by an edge. 
Each edge represents an interaction between the vertices.
Denote this graph by $K_{np_1,np_2,\dots, np_m}$ which has $n$ vertices and $m$ partitions where $\sum_{i=1}^m p_i =1$ and $p_i>0$ for $i=1,2,\dots, m$.
We fix the parameters $m$ and $p_i$'s hereafter.
Without loss of generality, we assume $p_1 \leq p_2 \leq \dots \leq p_m$.
We may also assume that $np_i \in \mathbb N$ for every $i$ so that $K_{np_1,np_2,\dots, np_m}$ is well defined whenever such considerations are required.
Let $V = \bigcup_{i=1}^m J_i$ be the set of all vertices where $J_i$ denotes the set of the $i$th partition of the vertices. 
Note $np_i = |J_i|$. 

We define $\Omega_i \colonequals \{\pm 1\}^{J_i}$ for $i=1,\dots,m$ so that $\Omega= \prod_{i=1}^m \Omega_i$ is our configuration space.
Each configuration $\sigma \in \Omega$ has a unique representation $(\sigma^{(1)},\dots,\sigma^{(m)}) \in \prod_{i=1}^m \Omega_i$ and both representations are understood throughout this paper.

For each $\sigma \in \Omega$, define the \textit{magnetization} on $J_i$ by $S^{(i)}(\sigma) \colonequals \sum_{v \in J_i} \sigma(v) /n$, $i=1,\dots,m$.
For the Markov chain $(\sigma_t)_{t \geq 0}=(\sigma_t^{(1)},\dots,\sigma_{t}^{(m)})_{t\geq 0}$ starting at $\sigma=(\sigma^{(1)},\dots,\sigma^{(m)}) \in \prod_{i=1}^m \Omega_i$, we define the corresponding magnetization on $J_i$ by \[S_{t}^{(i)} \colonequals \frac{1}{n}\sum_{v\in V_i}\sigma_{t}^{(i)}(v) \ \text{for}\ i\in \{1,\dots,m\}, \ t\geq 0.\]
We sometimes use the vector notation $\mathbf S_t \colonequals (S^{(1)}_t, \dots, S^{(m)}_{t})$ for $t\geq 0$.
We call the process $(\mathbf S_t)_{t\geq 0}$ a \emph{magnetization chain}.
\thref{magmarkov} shows that $(\mathbf S_t)_{t\geq 0}$ is in fact a Markov chain. 
Note that it is a projection of the whole Markov chain $({\sigma_{t}})_{t\geq 0}$, so mixing of the whole chain $({\sigma_{t}})_{t\geq 0}$ implies the mixing of the chain $(\mathbf S_t)_{t\geq 0}$.
Our aim is to show the converse in a certain sense.

\subsection{Main results}
Given the above definitions and notations, our main result establishes the cutoff phenomenon on complete multipartite graphs.

\begin{thm}[Main result]
    For $m\in\mathbb N$ and $p_i>0$ such that $\sum_{i=1}^m p_i =1$, the Glauber dynamics for the Ising model on the complete multipartite graph $K_{np_1,\dots, np_m}$ exhibits a cutoff at $\frac{1}{2(1-\beta/\beta_{cr})} n\ln n$ with window size $O(n)$ in the high temperature regime $\beta<\beta_{cr}$ where  $\beta_{cr}=\beta_{cr}(p_1,\dots,p_m)$ is a constant defined in equation \eqref{betacrdef}.
\end{thm}

\begin{thm}\thlabel{lowtempresult}
In the low temperature regime $\beta > \beta_{cr}$, the dynamics is exponentially slow mixing, i.e., $t_{\mathrm{mix}}\geq C_1 \exp(C_2 n)$ for some constants $C_1,\ C_2 >0$ not depending on $n$. 
\end{thm}

A few remarks are in order.
Our main result is obtained as a consequence of \thref{upperbound} and \thref{lowerbound}. 
In the low temperature regime $\beta>\beta_{cr}$, the mixing time is exponentially slow, therefore identifying the critical temperature $\beta_{cr}$. 
In the $m=1$ case, there are no spin interactions so the chain is equivalent to the lazy random walk on an $n$-dimensional hypercube, which has a cutoff at $(n\ln n) /2 $ with window size $O(n)$ (see \parencite[]{hypercube} or \parencite[Chapter 18]{Peres}).
This result can be seen as a consequence of our main result since $m=1$ implies $\beta_{cr}=\infty$  (see equation \eqref{betacrdef}).

\subsection{Organization of the article}
As mentioned earlier, our proof is based on the ideas of \textcite{levin}.
We assume high temperatures until Section \ref{section6}.
We first observe that the magnetization chain is a Markov chain in its own right (\thref{magmarkov}).
A suitable scaling of the magnetization chain leads to a contraction property (\thref{normcontraction}).
This in turn gives a uniform variance bound of magnetizations in time (Sections \ref{section2} and \ref{section3}).
In Section \ref{section4}, we construct a coupling of the magnetization chain so that it couples in $\frac{1}{2(1-\beta/\beta_{cr})} n\ln n+O(n)$ steps with high probability.
After the magnetization coupling phase, by considering the "$2m$-coordinate chain" inspired by \cite{levin}, we can construct a post magnetization coupling to reach the full-mixing in another $O(n)$ steps.
This proves the upper bound (\thref{upperbound}).
We construct a suitable distinguishing-statistic of the magnetization chain \cite[see][Chapter 7.3]{Peres} to obtain the lower bound (\thref{lowerbound}). 
These upper and lower bound results establish the cutoff in the high temperature regime.
Exponentially slow mixing in the low temperature regime is shown in Section \ref{section6}.

\section{Contraction of the magnetization chain in high temperatures}\label{section2}
We describe the \emph{monotone coupling}.
Let $I$ and $U$ be independent uniform random variables over $V$ and $[0,1]$, respectively.
We consider the collection of Markov chains with starting configurations $\sigma \in \Omega$. 
Simultaneously define the next configurations at time $t=1$ by \[\sigma_{1} (i)=  \begin{cases}  \sigma (i) & \text{if $I \neq i$} \\ \mathbbm 1_{U < r_+(\sum_{j\neq k}S^{(j)}(\sigma))} -\mathbbm 1_{U \geq r_+(\sum_{j\neq k}S^{(j)}(\sigma))}  & \text{if $I=i \in J_k$}  \end{cases} \quad  \]
where $r_+$ is defined in equation \eqref{r+}.
Repeat this procedure independently for each time.
It is clear that each Markov chain $(\sigma_t)_{t \geq 0}$ above is a version of the Glauber dynamics on the complete multipartite graph with starting state $\sigma$'s, defined on a common probability space.
The above coupling is called a \emph{monotone coupling} in the sense that if $\sigma \leq \tilde \sigma$ are starting states for $(\sigma_t)_{t\geq 0}$ and $(\tilde \sigma _t)_{t\geq 0}$, respectively, then $S^{(i)}(\sigma) \leq S^{(i)} (\tilde \sigma )$ for $i=1,\dots,m$ so that $\sigma_{1} \leq \tilde \sigma_{1}$, and $\sigma_t \leq \tilde \sigma_t$ for any $t\geq 0$ accordingly. 

Define \[\mathcal S \colonequals \prod_{i=1}^m\{-p_i, -p_i+2/n,\dots, p_i\}.\]
\begin{prop}[Magnetization chain]\thlabel{magmarkov}
The process $(S^{(1)}_t, \dots, {S^{(m)}_{t}})_{t \geq 0}$ is a \\Markov chain on the magnetization state space $\mathcal S$. 
\end{prop}
\begin{proof}
Note that \begin{align*}
\mathbb P\bigl(({S^{(1)}_{t+1}},\dots, {S^{(m)}_{t+1}} )=(S^{(1)}_{t}  -\frac{2}{n},\dots, {S^{(m)}_{t}})\bigr)&= p_1 \frac{n{S^{(1)}_{t}} +|J_1|}{2|J_1|} r_- \Bigl(\sum_{j\neq 1}{S^{(j)}_{t}}\Bigr)
\\&=\frac{p_1+{S^{(1)}_{t}}}{2} r_-\Bigl(\sum_{j\neq 1}{S^{(j)}_{t}}\Bigr)\end{align*} is measurable with respect to the $\sigma$-algebra generated by $({S^{(1)}_{t}},\dots, {S^{(m)}_{t}})$.
Other cases can be dealt with similarly.
\end{proof}
\begin{rmk}
By symmetry, $({S^{(1)}_{t}},\dots, {S^{(m)}_{t}})$ starting from $\sigma$ and $(-{S^{(1)}_{t}},\dots, -{S^{(m)}_{t}})$ starting from $-\sigma$ have the same distributions. 
This can also be seen by the physical fact that the map $\sigma \mapsto -\sigma$ just corresponds to flipping the reference axis to which we are measuring the spins of each site.
This does not change the dynamics of the spin system.
\end{rmk}

\begin{defi}[Hamming distance]
For two configurations $\sigma$ and $\sigma '$, denote the Hamming distance by $\mathrm{dist}(\sigma, \sigma') \colonequals \frac{1}{2}\sum_{k \in V} |\sigma(k)-\sigma'(k)|$.
\end{defi}
\begin{rmk}
This is a metric on $\Omega$, which is equal to the number of sites with different spins for two configurations.
Similarly, we can define $\mathrm{dist}_i$ on $\Omega_i$, respectively, but $\mathrm{dist}_i$'s merely satisfy the triangle inequality.
\end{rmk}

\begin{lem}[Contraction in mean for monotone coupling]\thlabel{monotonecontraction}
For a monotone coupling $({\sigma_{t}}, {\sigma'_{t}} )_{t \geq 0}$ starting at $(\sigma,\sigma') = ((\sigma^{(1)},\dots,\sigma^{(2)}), (\sigma'^{(1)},\dots,\sigma'^{(2)} ))$, we have \[\begin{pmatrix}\mathbb E \mathrm{dist}_1({\sigma^{(1)}_t},{\sigma'^{(1)}_t})\\ \vdots \\  \mathbb E\mathrm{dist}_m({\sigma^{(m)}_{t}},{\sigma'^{(m)}_t}) \end{pmatrix} \leq \mathbf A^t \begin{pmatrix} \mathrm{dist}_1(\sigma^{(1)}, \sigma_0') \\ \vdots \\ \mathrm{dist}_m(\sigma^{(m)},\sigma'^{(m)}) \end{pmatrix}\] where \[\mathbf A=\mathbf A_n \colonequals \begin{pmatrix}a&b_1&b_1&\dots&b_1
\\ b_2 & a&b_2 &\dots&b_2 
\\b_3&b_3&a&\dots&b_3 
\\\vdots &\dots&&&\vdots
\\b_m &\dots&&\dots&a
\end{pmatrix}\] with 
$a\colonequals 1-1/n$, $b_k \colonequals p_k {\beta }/{n}$.
\end{lem}
\begin{proof}
Assume $d(\sigma,\sigma')=1$ with $-1=\sigma(v) = -\sigma'(v)$ for some vertex $v$.
Note $\sigma \leq \sigma'$.
Since we are considering a monotone coupling, it holds that for each $i=1,\dots,m$, \begin{align*}
\mathrm{dist}_i({\sigma^{(i)}_{1}},{\sigma'^{(i)}_{1}})= \mathbbm 1_{v\in J_i}(1-\mathbbm 1_{I=v}) + \mathbbm 1_{v \notin J_i}(\mathbbm 1_{I\in J_i} \mathbbm 1_{B_i})
\end{align*} where \[B_i=\Biggl\{r_+\biggl(\sum_{l\neq i}S_l(\sigma)\biggr)\leq U < r_+\biggl(\sum_{l\neq i}S_l(\sigma')\biggr)\Biggl\}.\]
Note that \begin{align*}
    \mathbb P(B_i) &=\frac{1}{2}\Biggl( \tanh\biggl(\beta \sum_{l\neq i}S_l(\sigma')\biggr) - \tanh\biggl(\beta \sum_{l\neq i}S_l(\sigma)\biggr)\Biggr)
    \\&=\frac{1}{2}\Biggl( \tanh\biggl(\beta \biggl(\sum_{l\neq i}S_l(\sigma)+\frac{2}{n}\biggr)\biggr) - \tanh\biggl(\beta \sum_{l\neq i}S_l(\sigma)\biggr)\Biggr) \mathbbm1_{v \notin J_i}
    \\ &\leq \tanh \frac{\beta}{n} \mathbbm1_{v \notin J_i}.
\end{align*} 
Since $I$ and $U$ are independent, for $i=1,\dots,m$, \begin{align*} \mathbb E\mathrm{dist}_i({\sigma^{(i)}_{1}},{\sigma'^{(i)}_{1}}) \leq \mathbbm 1_{v\in J_i}(1-\frac{1}{n})+\mathbbm 1_{v \notin J_i} p_i \tanh \frac{\beta}{n}.
\end{align*}

Suppose $\mathrm{dist}(\sigma,\sigma')=  k >1$.
There exists $\sigma^0 \colonequals \sigma$, $\sigma^{1}$, $\dots$, $\sigma^k \colonequals \sigma'$ such that $\mathrm{dist}(\sigma^i, \sigma^{i+1})=1$.
By the triangular inequality for $\mathrm{dist}_i$ and the fact $\tanh(\beta /n) \leq \beta/n$, \begin{align*}
\mathbb E\mathrm{dist}_i({\sigma^{(i)}_{1}},{\sigma'^{(i)}_{1}}) \leq (1-\frac{1}{n})\mathrm{dist}_i(\sigma^{(i)},\sigma'^{(i)})+ p_i \frac{\beta}{n} \sum_{ l\neq i }\mathrm{dist}_l(\sigma^{(l)},\sigma'^{(l)}).
\end{align*}

Furthermore, by the Markov property,
\begin{multline*}
    \mathbb E[\mathrm{dist}_i ({\sigma^{(i)}_{t+1}}, \sigma'^{(i)}_{t+1})| {\sigma_{t}}, {\sigma'_{t}} ] \leq (1-\frac{1}{n})\mathrm{dist}_i({\sigma^{(i)}_{t}}, {\sigma'^{(i)}_{t}}) + \frac{p_i \beta}{n} \sum_{l\neq i }\mathrm{dist}_l({\sigma^{(l)}_{t}},{\sigma'^{(l)}_{t}}).
\end{multline*}
By taking expectation and putting $x_{i,t} \colonequals \mathbb E\mathrm{dist}_i({\sigma^{(i)}_{t}}, {\sigma'^{(i)}_{t}})$, we have 
\begin{align*}
    \begin{pmatrix} x_{1,t} \\ \vdots \\ x_{m,t} \end{pmatrix} \leq \mathbf A \begin{pmatrix} x_{1,t-1} \\ \vdots \\ x_{m,t-1} \end{pmatrix}.
\end{align*} 
Iterating gives \[ \begin{pmatrix} x_{1,t} \\ \vdots \\ x_{m,t} \end{pmatrix} \leq \mathbf A^t \begin{pmatrix} \mathrm{dist}_1(\sigma^{(1)},\sigma'^{(1)}) \\ \vdots \\ \mathrm{dist}_m(\sigma^{(m)},\sigma'^{(m)}) \end{pmatrix}.\]
\end{proof}

From now on, $\mathbf A$ (which depends on the number of vertices $n$) always denotes the matrix defined in \thref{monotonecontraction}.
Note that $\mathbf A$ is a positive matrix, so by the Perron-Frobenius theorem, there exists the largest eigenvalue $g=g_n>0$ with the left eigenvector $\mathbf a^T\colonequals (a_1, \dots, a_m) >\mathbf 0$ normalized in $l^1$ norm. Note that $g$ has algebraic multiplicity $1$ (see \parencite[Section 8.2]{Meyer} for a proof), so $\mathbf a^T$ is unique.


We fix the following notations  \begin{gather} 
    \upsilon \colonequals n(1-g)  \text{ and }\label{upsilondef}
    \\ \beta_{cr} \colonequals \frac{1}{(m-1)\sum_{i=1}^m a_ip_i} \label{betacrdef}
\end{gather} 
where $g$ and $(a_1,\dots,a_m)$ are defined in the previous paragraph.
Another characterization of $\beta_{cr}$ is given in \thref{slowmixinglemma}.
Insuk Seo commented\footnote{personal communication} that it can also be characterized as the threshold value of $\beta$ that makes $\mathbf K$ positive definite where $\mathbf K$ is defined through the equation $\mathbf A=\mathbf I-\frac{1}{n}\mathbf K$, $\mathbf I$ being the $m$-by-$m$ identity matrix.
\thref{upsilon} connects the quantities $\upsilon$ and $\beta_{cr}$.

\begin{prop}\thlabel{upsilon}
    The left eigenvector $\mathbf a^T$ only depends on $p_1,\dots,p_m$.
    Moreover, $\upsilon$ only depends on $p_1,\dots,p_m,$ and $\beta$ through the following equation: \[\upsilon = 1-\beta(m-1)\sum_{i=1}^ma_ip_i.\]
    Therefore, $\beta_{cr}$ only depends on $p_1,\dots,p_m$, and we have $\upsilon = 1-\beta/\beta_{cr}$. 
\end{prop}
\begin{proof}
    Since $g$ satisfies \begin{align*}
    0&=(n/\beta)^m\det(\mathbf A-gI)=\det (n \mathbf A/\beta -n gI/\beta) \\&= \det \begin{pmatrix}\frac{\upsilon-1}{\beta}&p_1&p_1&\dots&p_1
    \\ p_2 & \frac{\upsilon-1}{\beta}&p_2 &\dots&p_2 
    \\p_3&p_3&\frac{\upsilon-1}{\beta}&\dots&p_3
    \\\vdots &\dots&&&\vdots
    \\p_m &\dots&&\dots&\frac{\upsilon-1}{\beta}
    \end{pmatrix},
    \end{align*}
    it holds that $(\upsilon-1)/\beta$ is a root of a polynomial with coefficients only depending on $p_1,\dots,p_m$.
    Since $\mathbf a$ is in the kernel of the transpose of the above matrix, it only depends on $p_1,\dots,p_m$.
    
    Finally, $g= \|\mathbf A^T \mathbf a\|_1= 1-1/n+\frac{\beta}{n}(m-1) \sum_{k}a_ip_i$ implies $\upsilon =1 -\beta (m-1)\sum_i a_ip_i$.
    
\end{proof}

We collect further properties of the matrix $\mathbf A$ and its left eigenvector $\mathbf a^T$ in the next two lemmas. 

\begin{lem}\thlabel{finalfinallemma}
We have \[a_1\geq \dots\geq a_m \ \text{and}\ \sum_{i=1}^ma_ip_i \leq \frac{1}{m}.\]
The equality in the latter holds if and only if $p_1=\dots=p_m$.
\end{lem}
\begin{proof}
Recall that we assumed $p_1\leq \dots \leq p_m$.

We claim that $a_1\geq \dots\geq a_m$.
To that end, fix $i<j$.
From $\mathbf a^T\mathbf A=g\mathbf a^T$, we have $(1-\frac{1}{n})a_i +\frac{\beta}{n}\sum_{k\neq i}a_kp_k-ga_i=0=(1-\frac{1}{n})a_j +\frac{\beta}{n}\sum_{k\neq j}a_kp_k-ga_j$.
Then $(1-\frac{1}{n}-g-\frac{\beta p_i}{n})a_i=(1-\frac{1}{n}-g-\frac{\beta p_j}{n})a_j$, i.e., $(\beta p_i +1-\upsilon )a_i=(\beta p_j +1-\upsilon )a_j$.
Thus, $p_i\leq p_j$ implies $a_i\geq a_j$, proving the claim.

By Chebyshev's sum inequality, since $a_i\geq a_j$ and $p_i\leq p_j$ whenever $i<j$, \begin{align*}
    \sum_{i=1}^ma_ip_i \leq \frac{1}{m} \left(\sum_{i=1}^m a_i\right)\left(\sum_{i=1}^m p_i\right)=\frac{1}{m}.
\end{align*}
The equality holds if and only if $a_1= \dots =a_m$ or $p_1= \dots =p_m$.
The proof is now complete by noticing the fact that $(\beta p_i +1-\upsilon )a_i=(\beta p_j +1-\upsilon )a_j$ and $a_1= \dots =a_m=1/m$ imply $p_1= \dots =p_m$.
\end{proof}
\begin{rmk}
As a consequence, we obtain a lower bound $\beta_{cr}\geq m/(m-1)$.
\end{rmk}

\begin{lem}\thlabel{newlemma}
For $\mathbf 0 \leq \mathbf s \in \mathcal S$ and $\mathbf p \colonequals (p_1,\dots, p_m)^T$, we have \[\|\mathbf A^t\mathbf s\|_1 \leq  g^t \left(\sum_{i=1}^m \frac{(s^{(i)})^2}{p_i}\right)^{1/2},\quad \mathbf e_j^T \mathbf A^t \mathbf s \leq \sqrt{p_j} g^t  \left(\sum_{i=1}^m \frac{(s^{(i)})^2}{p_i}\right)^{1/2} . \]
In particular, it holds that \[\|\mathbf A^t \mathbf{p}\|_1 \leq g^t, \quad \mathbf e_j^T \mathbf A^t \mathbf p \leq \sqrt{p_j}g^t. \]
\end{lem}
\begin{proof}
We want to find a symmetric matrix $\mathbf{C}$ which is similar to $\mathbf A$.
To that end, suppose that there exists an invertible diagonal matrix $\mathbf{D}=\mathrm{diag}(d_1, \dots, d_m)$ and a symmetric matrix $\mathbf{C}$ such that $\mathbf{C}=\mathbf{D}^{-1}\mathbf A\mathbf{D}$. 
Then $\mathbf{D}\mathbf A^T\mathbf{D}^{-1}=\mathbf{C}^T=\mathbf{C}=\mathbf{D}^{-1}\mathbf A\mathbf{D}$, so $\mathbf{D}^2\mathbf A^T=\mathbf A\mathbf{D}^2$, which leads to $d_i^2 p_j = p_i d_j^2$ for $i,j \in \{1,2,\dots, m\}$.
With the above in mind, let $\mathbf{D}\colonequals \mathrm{diag}(\sqrt{p_1},\dots ,\sqrt{p_m})$ and $\mathbf{C}\colonequals (c_{ij}) $ where $c_{ii}=1-1/n$ and $c_{ij}=\sqrt{p_ip_j}{\beta}/{n}$ for $i\neq j$.
Note that $\mathbf{C}$ is real-symmetric and $\mathbf{C}=\mathbf{D}^{-1}\mathbf A\mathbf{D}$. 
Then, by the spectral theorem for real symmetric matrices, $\|\mathbf{C}\|_2=g$.
Note that $\mathbf{C}$ and $\mathbf A$ have the same real eigenvalues since they are similar.

Observe that for $\mathbf x, \mathbf y \in \mathbb R^m$, $\|\mathbf x\mathbf y^T\|_2=\|\mathbf x\|_2\|\mathbf y\|_2$. 
This can be easily checked by the equalities \[\|\mathbf x \mathbf y^T\|_2=\sup_{\|\mathbf z\|_2=1} \|\mathbf x \mathbf y^T \mathbf z\|_2 =\sup_{\|\mathbf z\|_2=1} |\mathbf y^T \mathbf z |\|\mathbf x \|_2 = \|\mathbf y\|_2 \|\mathbf x\|_2 .\]

Let $\mathbbm 1 \colonequals (1,\dots,1)^T$.
The case $\mathbf s=\mathbf 0$ is trivial, so assume $\mathbf s>\mathbf 0$.
Since $\mathbf s \mathbbm 1^T$ has rank 1, $\mathbf{C}^t\mathbf{D}^{-1}\mathbf s\mathbbm 1^T \mathbf{D}$ has rank $1$ .
Also, its elements are positive, so it has a positive eigenvalue by the Perron-Frobenius theorem.
Thus, $\mathrm{Tr}(\mathbf{C}^t\mathbf{D}^{-1}\mathbf s\mathbbm 1^T \mathbf{D})$ is equal to its spectral radius, from which the following inequality follows:
\begin{align*}
    \|\mathbf A^t \mathbf s\|_1 &= \mathbbm 1^T \mathbf A^t \mathbf s = \mathbbm 1^T \mathbf{D}\cdot \mathbf{C}^t\mathbf{D}^{-1} \mathbf s  =\mathrm{Tr}(\mathbf{C}^t\mathbf{D}^{-1}\mathbf s\mathbbm 1^T \mathbf{D}) \leq \|\mathbf{C}^t\mathbf{D}^{-1}\mathbf s\mathbbm 1^T \mathbf{D} \|_2 
    \\&\leq \|\mathbf{C}\|_2^t \|\mathbf{D}^{-1}\mathbf s \mathbbm 1^T \mathbf{D}\|_2 = g^t \|\mathbf{D}^{-1}\mathbf s\|_2 \|\mathbf{D} \mathbbm 1 \|_2 = g^t \left(\sum_{i=1}^m \frac{(s^{(i)})^2}{p_i}\right)^{1/2} .
\end{align*}
Similarly, \[\mathbf e_j^T \mathbf A^t \mathbf s \leq \|\mathbf{C}\|_2^t \|\mathbf{D}^{-1}\mathbf s \mathbf e_j^T \mathbf{D}\|_2 =g^t\|\mathbf{D}^{-1}\mathbf s\|_2\|\mathbf{D}\mathbf e_j\|_2= \sqrt{p_j}g^t\Biggl(\sum_{i=1}^m \frac{(s^{(i)})^2}{p_i}\Biggr)^{1/2}.\]
\end{proof}
\begin{rmk}
Another relatively simple proof of $\|\mathbf A^t\mathbf s\|_1 \leq  g^t \left(\sum_{i=1}^m \frac{(s^{(i)})^2}{p_i}\right)^{1/2}$ can be given as follows. 
By the Cauchy-Schwartz inequality, we have $\sqrt{\sum_i s_i^2 /p_i}\geq \sum_i s_i$.
Then $\|\mathbf A^t \mathbf s\|_1\leq \|\mathbf D^{-1}\mathbf A^t \mathbf s\|_2=\|\mathbf D^{-1}\mathbf A^t \mathbf D \mathbf D^{-1} \mathbf s\|_2=\|\mathbf C^t \mathbf D^{-1}\mathbf s\|_2\leq g^t \|\mathbf D^{-1} \mathbf s\|_2$.
\end{rmk}

From now on, for brevity, we use the notation \[\mathbf p \colonequals (p_1,\dots,p_m)^T.\]

\begin{lem}\thlabel{lemma}
For a monotone coupling $({\sigma_{t}},{\sigma'_{t}})_{t\geq 0}$ starting at $(\sigma,\sigma')$, we have
\[\mathbb E \sum_{i=1}^m a_i \mathrm{dist}_i (\sigma_t,\sigma'_t) \leq  g^t \sum_{i=1}^m a_i \mathrm{dist}_i (\sigma,\sigma').  \]
Moreover, for $i=1,\dots,m$, \[\mathbb E \mathrm{dist}_i ({\sigma^{(i)}_t}, {\sigma'^{(i)}_t}) \leq n\sqrt{p_i}g^t.\]
\end{lem}
\begin{proof}
From \thref{monotonecontraction}, \begin{align*}
\mathbb E \sum_{i=1}^m a_i \mathrm{dist}_i (\sigma_t,\sigma'_t)  &= \mathbf a^T \begin{pmatrix}\mathbb E\mathrm{dist}_1({\sigma^{(1)}_t},{\sigma'^{(1)}_t})\\ \vdots \\  \mathbb E\mathrm{dist}_m({\sigma^{(m)}_{t}},{\sigma'^{(m)}_{t}}) \end{pmatrix} \leq  \mathbf a^T \mathbf A^t \begin{pmatrix} \mathrm{dist}_1(\sigma^{(1)}, \sigma'^{(1)}) \\ \vdots \\ \mathrm{dist}_m(\sigma^{(m)},\sigma'^{(m)}) \end{pmatrix} 
\\ &\leq g^t \mathbf a^T \begin{pmatrix} \mathrm{dist}_1(\sigma^{(1)}, \sigma'^{(1)}) \\ \vdots \\ \mathrm{dist}_m(\sigma^{(m)},\sigma'^{(m)}) \end{pmatrix} \leq g^t \sum_{i=1}^m a_i \mathrm{dist}_i (\sigma,\sigma').
\end{align*}

Notice that $\mathrm{dist}_k({\sigma^{(k)}_{t}},{\sigma'^{(k)}_{t}}) \leq np_k$ for each $k$, so \thref{newlemma} implies \[
    \mathbb E \mathrm{dist}_i ({\sigma^{(i)}_t}, {\sigma'^{(i)}_t}) \leq n\mathbf e_i^T \mathbf A^t \mathbf p \leq n\sqrt{p_i}g^t.
\]
\end{proof}

We would like to translate \thref{lemma} to the case of magnetization chains, which is done in \thref{normcontraction}.

\begin{lem}\thlabel{generalcontraction}
For starting magnetizations $\mathbf s = (s^{(1)},\dots,s^{(m)})\geq(s'^{(1)},\dots, s'^{(m)})= \mathbf s'$, the magnetization chains satisfy
\[\mathbf 0  \leq \begin{pmatrix}\mathbb E_{\mathbf s}{S^{(1)}_{t}} -\mathbb E_{\mathbf s'}{S'^{(1)}_{t}}\\ \vdots \\  \mathbb E_{\mathbf s}{S^{(m)}_{t}} -\mathbb E_{\mathbf s'}{S'^{(m)}_{t}} \end{pmatrix} \leq \mathbf A^t \begin{pmatrix} s^{(1)}-s'^{(1)} \\ \vdots \\ s^{(m)}-s'^{(m)} \end{pmatrix}.\]
\begin{rmk}
We say such pairs of starting magnetizations are \emph{monotone pairs}.
\end{rmk}
\end{lem}
\begin{proof}
Let $({\sigma_{t}},{\sigma'_{t}})$ be a monotone coupling starting from $(\sigma,\sigma')$ where $\sigma \geq \sigma'$ and $S^{(i)}(\sigma)=s_i$, $S'^{(i)}(\sigma')=s_i'$ for $i=1,\dots,m$.
Such a monotone coupling exists because of the given condition $s_i \geq s_i'$ for each $i$.
Since $\sigma_i \geq \sigma_i'$, we have $s_i-s_i' =\frac{2}{n}\mathrm{dist}_i(\sigma_i,\sigma_i')$ for each $i$.
By monotonicity, ${\sigma^{(i)}_t} \geq {\sigma'^{(i)}_t}$ for each $i$.
Thus, ${S^{(i)}_{t}}-{S'^{(i)}_{t}}=|{S^{(i)}_{t}}-{S'^{(i)}_{t}}|=\frac{2}{n}\mathrm{dist}_i({\sigma^{(i)}_t},{\sigma'^{(i)}_t})\geq0$ for each $i$.
Then, by \thref{monotonecontraction}, \[\mathbf 0  \leq \begin{pmatrix}\mathbb E_{\sigma}{S^{(1)}_{t}} -\mathbb E_{\sigma'}{S'^{(1)}_{t}}\\ \vdots \\  \mathbb E_{\sigma}{S^{(m)}_{t}} -\mathbb E_{\sigma'}{S'^{(m)}_{t}} \end{pmatrix} = 
\begin{pmatrix}\mathbb E_{\sigma,\sigma'}|{S^{(1)}_{t}}-{S'^{(1)}_{t}}|\\ \vdots \\  \mathbb E_{\sigma,\sigma'}|{S^{(m)}_{t}}-{S'^{(m)}_{t}}| \end{pmatrix} \leq \mathbf A^t \begin{pmatrix} s^{(1)}-s'^{(1)} \\ \vdots \\ s^{(m)}-s'^{(m)} \end{pmatrix}.\]
Now, we can complete the proof since we have $\mathbb E_{\sigma}{S^{(i)}_{t}}-\mathbb E_{\sigma'}{S'^{(i)}_{t}}= \mathbb E_{\mathbf s}{S^{(i)}_{t}} -\mathbb E_{\mathbf s'}{S'^{(i)}_{t}}$ for each $i$ by \thref{magmarkov}.
\end{proof}

Recall that $\circ$ denotes a Hadamard product.
\begin{prop}\thlabel{normcontraction}
For a monotone coupling $(\sigma_t,\sigma_t')_{t\geq 0}$ starting at $(\sigma,\sigma')$ with magnetizations $(\mathbf s, \mathbf s')$, we have \begin{align*}
\mathbb E_{\sigma, \sigma'}\|\mathbf a \circ \mathbf S_t - \mathbf a \circ \mathbf S_t'\|_1
 \leq g ^t \|\mathbf a \circ \mathbf s - \mathbf a \circ \mathbf s'\|_1
  .\end{align*}
 Moreover, not depending on the coupling, we have
 \begin{align*}
 \|\mathbb E_{\mathbf s}\mathbf a \circ \mathbf S_t - \mathbb E_{\mathbf s'}\mathbf a \circ \mathbf S_t'\|_1
 \leq g ^t \|\mathbf a \circ \mathbf s - \mathbf a \circ \mathbf s'\|_1.
 \end{align*}
\end{prop}
\begin{proof}
For any magnetizations $\mathbf s \equiv \mathbf s_{(0)}$ and $\mathbf s' \equiv \mathbf s_{(m)}$, there exists $\mathbf s _{(1)} ,\dots, \mathbf s_{(m-1)} \in \mathcal S \subset \mathbb R^m$ such that  $\mathbf s_{(i-1)} - \mathbf s_{(i)} = \mathbf e_i (s^{(i)}-s'^{(i)})$ for $i=1,\dots, m$.
In particular, $\mathbf s_{(i-1)}$ and $\mathbf s_{(i)}$ are a monotone pair for each $i$.
Then we can consider a monotone coupling $(\sigma_{(0),t},\dots,\sigma_{(m),t})_{t\geq 0}$ with starting states $(\sigma_{(0)},\dots ,\sigma_{(m)})$ such that $\sigma_t=\sigma_{(0),t}$, $\sigma_t'=\sigma_{(m),t}$ for $t\geq 0$, and the magnetization of the starting configuration $\sigma_{(i)}$ is $\mathbf s_{(i)}$ for $i=0,\dots,m$.

Let $\mathbf S_{(j),t}$ be the magnetization chain corresponding to $\sigma_{(j),t}$ for $j=0,\dots,m$.
By telescoping, \thref{generalcontraction} gives \begin{align*}
    &\mathbb E_{\sigma,\sigma'} \|\mathbf a \circ \mathbf S_{t} - \mathbf a \circ \mathbf S_{t}'\|_1  \leq \sum_{j=1}^m \mathbb E_{\sigma_{(j-1)},\sigma_{(j)}} \| \mathbf a \circ \mathbf S_{(j-1),t} - \mathbf a \circ \mathbf S_{(j),t}\|_1
    \\&\leq \sum_{j=1}^m \mathbf a^T \mathbf A^t \mathbf {e}_j |s^{(j)}-s'^{(j)} | = g^t \sum_{j=1}^m a_j |s^{(j)}-s'^{(j)} | = g ^t \|\mathbf a \circ \mathbf s - \mathbf a \circ \mathbf s'\|_1. 
\end{align*}
Then, the triangle inequality and \thref{magmarkov} imply \[\|\mathbb E_{\mathbf s}\mathbf a \circ \mathbf S_t - \mathbb E_{\mathbf s'}\mathbf a \circ \mathbf S_t'\|_1
 \leq g ^t \|\mathbf a \circ \mathbf s - \mathbf a \circ \mathbf s'\|_1.\]
\end{proof}

\section{Variance bound of the magnetization in high temperatures} \label{section3}
The next lemma is a generalization of Lemma 2.6 in \cite{levin} to Markov chains with a finite state space in $\mathbb R^m$.
Observe that for square-integrable $\mathbb R^m$-valued i.i.d. random vectors $X,Y$, we have $\mathbb{V}\mathrm{ar}X=\frac{1}{2}\mathbb E\|X-Y\|_2^2$.
\begin{lem}\thlabel{variancebound}
Let $(\mathbf Z_t)_{t\geq 0}$ be a Markov chain in a finite state space $\tilde {\mathcal S} \subseteq \mathbb R^m$. 
Suppose that there exists $0<r<1$ such that for any $\theta, \theta' \in \tilde {\mathcal S}$, \[\|\mathbb E_\theta \mathbf Z_t - \mathbb E_{\theta'}\mathbf Z_t'\|_1 \leq r^t \|\theta-\theta'\|_1.\] 
Then, for the $l^2$ norm variance, \[\sup_{\theta \in \mathcal S} \mathbb{V}\mathrm{ar}_{\theta} \mathbf Z_t \leq m\sup_{\theta \in \mathcal S} \mathbb{V}\mathrm{ar}_{\theta}\mathbf Z_1 \: \min \{t, (1-r ^2)^{-1}\}.\]
\end{lem}
\begin{proof}
Put $v_t \colonequals \sup_{\theta \in \mathcal S} {\mathbb{V}\mathrm{ar}}_{\theta}\mathbf Z_t$.
Let $(\mathbf Z_t)$ and $(\mathbf Z_t')$ be independent copies of the chain starting from $\theta \in \tilde{\mathcal S}$.
The idea is to condition on the first step.
Note that $\|\mathbf x\|_2 \leq \|\mathbf x\|_1 \leq  \sqrt{m}\|\mathbf x\|_2$ for $\mathbf x \in \mathbb R^m$.
Then by the observation right before the statement of this lemma, \[ \frac{1}{2}\mathbb E_{\theta}\|\mathbf Z_1-\mathbf Z_1'\|_1^2 \leq m\frac{1}{2}\mathbb E_{\theta}\|\mathbf Z_1-\mathbf Z_1'\|_2^2 \leq mv_1.\]
By the assumption and Markov property, we have \[\|\mathbb E_\theta[\mathbf Z_t | \mathbf Z_1] - \mathbb E_\theta[\mathbf Z_t' | \mathbf Z_1']\|_1=\|\mathbb E_{\mathbf Z_1}[\mathbf Z_{t-1}]-\mathbb E_{\mathbf Z_1'}[\mathbf Z_{t-1}']\|_1 \leq r^{t-1}\|\mathbf Z_1-\mathbf Z_1'\|_1.\]
Thus, for $\theta \in \tilde{\mathcal S}$, \begin{align*}
    \mathbb{V}\mathrm{ar}_\theta [\mathbb E_\theta (\mathbf Z_t|\mathbf Z_1)]&= \frac{1}{2}\mathbb E_\theta \|\mathbb E_{\mathbf Z_1} \mathbf Z_{t-1}-\mathbb E_{\mathbf Z_1'} \mathbf Z_{t-1}'\|_2^2  \leq \frac{1}{2}\mathbb E_\theta \|\mathbb E_{\mathbf Z_1} \mathbf Z_{t-1}-\mathbb E_{\mathbf Z_1'} \mathbf Z_{t-1}'\|_1^2
    \\ &\leq \frac{1}{2}\mathbb E_\theta \Big[r^{2(t-1)}\|\mathbf Z_1-\mathbf Z_1'\|_1^2\Big] \leq  mv_1 r^{2(t-1)}.
\end{align*}
By the Markov property, for every $\theta \in \tilde{\mathcal S}$, $\mathbb{V}\mathrm{ar}_\theta [\mathbf Z_t | \mathbf Z_1 ] \leq v_{t-1}$, so \[\sup_{\theta \in \mathcal S}\mathbb E_\theta [\mathbb{V}\mathrm{ar}_\theta [\mathbf Z_t|\mathbf Z_1]] \leq v_{t-1}.\]
The total variance formula holds since we are using the $l^2$ norm.
Thus, taking supremum over $\theta \in \tilde{\mathcal S}$ in the total variance formula $\mathbb{V}\mathrm{ar}_\theta \mathbf Z_t= \mathbb E_\theta \big[\mathbb{V}\mathrm{ar}_\theta [\mathbf Z_t|\mathbf Z_1]\big] + \mathbb{V}\mathrm{ar}_\theta \big[\mathbb E_\theta [\mathbf Z_t|\mathbf Z_1]\big]$, we have $v_t \leq v_{t-1}+mv_1r^{2(t-1)}$.
Upon iterating,   \[v_t \leq mv_1 \sum_{t=1}^{t} r^{2(t-1)} \leq mv_1 \min\big\{t, (1-r^2)^{-1}\big\}.\]
\end{proof}

The following proposition is an important result bounding the variance of magnetization chains uniformly in time.
\begin{prop}\thlabel{magvariancebound}
Let $\beta<\beta_{cr}$.
For an arbitrary starting configuration $\mathbf s$ and $t \geq 0$, we have \[\sum_{i=1}^m\mathbb{V}\mathrm{ar}_{\mathbf s}({S^{(i)}_{t}})=C/n\] where $C>0$ only depends on $p_1,\dots,p_m$, and $\beta$.
\end{prop}
\begin{proof}
Observe that $\sum_{i=1}^m  \mathbb{V}\mathrm{ar}_{\mathbf s}(a_i S^{(i)}_{t}) = \mathbb{V}\mathrm{ar}_{\mathbf s}(\mathbf a \circ \mathbf S_t)$.
Note that increments of $\mathbf S_t$ are bounded by $2/n$ in absolute value.
Then, from \thref{finalfinallemma}, we have \[\sum_{i=1}^m \mathbb{V}\mathrm{ar}_{\mathbf s}{a_i S^{(i)}_{1}} \leq  a_1^2 (2/n)^2.\]
By \thref{finalfinallemma}, \thref{normcontraction}, and \thref{variancebound}, we have
\[a_m^2\sum_{i=1}^m\mathbb{V}\mathrm{ar}_{\mathbf s}({S^{(i)}_{t}}) \leq \sum_{i=1}^m\mathbb{V}\mathrm{ar}_{\mathbf s}(a_i{S^{(i)}_{t}}) \leq m\frac{4a_1^2}{n^2}\frac{1}{1-g^2} = \frac{4ma_1^2}{\upsilon n(1+g)}\leq \frac{4ma_1^2}{\upsilon n}. \]
Note that \thref{upsilon} assures $\upsilon>0$.
\end{proof}

We also establish a bound for the expected magnetization on subsets of partitions.
To that end, we need the following observation.
\begin{lem}\thlabel{zerospin}
    For each $i \in V$, $\mathbb E_{\mu}(\sigma(i))=0$  where $\mu$ is the Gibbs distribution.
    In particular, we have $\mathbb E_\mu (S^{(i)})=0$.
    \end{lem}
    \begin{proof}
    Since $\mu(\sigma)=\mu(-\sigma)$ for each configuration $\sigma$ and $\sigma \mapsto -\sigma$ is a bijection from $\Omega$ into itself, we have $\mathbb E_{\mu}(\sigma(i))= \sum_\sigma \sigma(i) \mu(\sigma) = \sum _{\sigma : \sigma(i)=1} \mu(\sigma)-\sum _{\sigma : \sigma(i)=-1} \mu(\sigma)=0$. 
\end{proof}

\begin{prop}[Expected magnetization bound]\thlabel{expectedmagbound}
Let $\beta< \beta_{cr}$ and $1\leq i \leq m$. For any $B \subseteq J_i$ and a chain $(\sigma_t)_{t \geq 0}$ starting at $\sigma \in \Omega$, define $M_t (B) \colonequals \frac{1}{2}\sum_{k \in B}{\sigma_{t}} (k)$.
Then \[|\mathbb E_\sigma M_t(B)| \leq |B| g^t /\sqrt{p_i}.\]
Furthermore, for $t \geq \frac{1}{2(1-\beta/\beta_{cr})} n\ln n$, we have \[\mathbb{V}\mathrm{ar}_\sigma(M_t (B)) =O(n)\ , \quad \mathbb E_\sigma|M_t (B)| =  O(\sqrt{n}).\]  
\end{prop}
\begin{proof}
Let "+" denote the configuration such that all spins are $1$ and "-" denote the configuration with all spins $-1$.
Let $(\sigma_{t}^+, \sigma_{t}^\mu  ,\sigma_{t} ^-)$ be a monotone coupling with starting configuration $(+,\mu,-)$ where $\mu$ is the stationary distribution.
Let $i \in \{1,\dots,m\}$.
By \thref{lemma} and \thref{zerospin}, \begin{align*}
    \mathbb E_+[M_t(J_i)^+] \leq \mathbb E_{+,\mu}|M_t(J_i)^+ -M_t(J_i)^\mu| +\mathbb E_\mu [M_t(J_i)^\mu] \leq n\sqrt{p_i} g^t.
\end{align*}
Then, by symmetry, for $v \in J_i$, $\mathbb E_+[M_t(v)] \leq n\sqrt{p_i}g^t/|J_i| =g^t/\sqrt{p_i}$.
Thus, by summing over sites in $B$,   $\mathbb E_+[M_t(B)^+] \leq |B|g^t/\sqrt{p_i}$.
However, for any configuration $\sigma$, by monotonicity, $\mathbb E_+[M_t(B)^+] \geq \mathbb E_\sigma [M_t(B)] \geq \mathbb E_-[M_t(B)^-]$.
Considering the remark after \thref{magmarkov}, $\mathbb E_- [M_t(B)^-]= -\mathbb E_+ [M_t(B)^+]$.
Thus, $|\mathbb E_\sigma [M_t(B)]| \leq |\mathbb E_+[M_t(B)^+]| \leq |B|g^t/\sqrt{p_i}$ for any $\sigma$.

Now, by \thref{magvariancebound}, $ O(1/n)=\mathbb{V}\mathrm{ar}{S^{(i)}_{t}} =\mathbb{V}\mathrm{ar}(M_t(J_i)2/n) $, so \[\mathbb{V}\mathrm{ar}_+( M_t(J_i) ) = O(n). \]
Thus, for $t \geq \frac{1}{2(1-\beta/\beta_{cr})} n\ln n$, \[\mathbb E_+(M_t(J_i)^2)=\mathbb{V}\mathrm{ar}_+(M_t(J_i))+(\mathbb E_+M_t(J_i))^2=O(n)\]
However, by symmetry, for any fixed $v_1,v_2 \in J_i$, \[\mathbb E_+(M_t(J_i)^{2})=np_i+\binom{np_i}{2}\mathbb E_+(\sigma_{t}^+(v_1)\sigma_{t}^+(v_2)). \]
Thus, \[|\mathbb E_+\sigma_{t}^+(v_1)\sigma_{t}^+(v_2)|=O(1/n).\]
Likewise, for $B\subseteq J_i$, \[\mathbb E_+(M_t(B)^2) =|B|+ \binom{|B|}{2}\mathbb E_+(\sigma_{t}^+(v_1)\sigma_{t}^+(v_2)) \leq O(n). \]
Similarly, $\mathbb E_-M_t(B)^2  \leq O(n)$, so from $(M_t(B))^2 \leq (M_t(B)^+)^2+(M_t(B)^-)^2$, \[\mathbb E(M_t(B)^2) =O(n) \] whenever $t \geq \frac{1}{2(1-\beta/\beta_{cr})} n\ln n$.
Thus, for $t \geq \frac{1}{2(1-\beta/\beta_{cr})} n\ln n$, \[\mathbb{V}\mathrm{ar}_\sigma(M_t(B))=O(n).\] 

Lastly, for $t \geq \frac{1}{2(1-\beta/\beta_{cr})} n\ln n$, from Jensen's inequality, \begin{align*}
     \mathbb E_\sigma|M_t (B)| &\leq \sqrt{\mathbb E_\sigma|M_t (B)| ^2} =\sqrt{(\mathbb E_\sigma[M_t (B)])^2 +\mathbb{V}\mathrm{ar}_\sigma(M_t (B))}
    \\  &\leq |\mathbb E_\sigma[M_t (B)]| + \sqrt{\mathbb{V}\mathrm{ar}_\sigma(M_t (B))}=O(\sqrt{n}).
\end{align*}
\end{proof}

\section{Couplings}\label{section4}
Fix the notation \[t_n \colonequals \frac{1}{2(1-\beta/\beta_{cr})} n \ln n.\]

\begin{defi}[Modified matching]
Let $\sigma\in \Omega$ and $\sigma'\in \Omega$ have magnetizations $\mathbf s \in \mathcal S$ and $\mathbf s'\in \mathcal S$, respectively.
Consider two copies of the graph, $V=\bigcup_i J_i$ and $V'=\bigcup_i J_i'$.
Let $i \in \{1,\dots,m\}$.
If $s^{(i)} \geq s'^{(i)}$, then it is possible to match each site in $J_i'$ with $+1$ spin to a site in $J_i$ with $+1$ spin.
Any leftover sites in $J_i'$ are arbitrarily matched to the leftover sites in $J_i$.
We match the sites in a similar way whenever $s^{(i)}\leq s'^{(i)}$.
This defines a bijection $f_{\sigma,\sigma'}\colon V \to V'$.

We call this bijection a \emph{modified matching of $\sigma$ and $\sigma'$}.
\end{defi}

\begin{defi}[Modified monotone update and coupling]
Let $f_{\sigma,\sigma'}\colon V \to V'$ be a modified matching of $\sigma,\sigma' \in\Omega$.
Let $I$ and $U$ be uniformly distributed over $V=\bigcup_{i=1}^m J_i$ and $[0,1]\subseteq \mathbb R$, respectively, and be independent.
Suppose $I \in J_\eta$ for some $\eta \in \{1,\dots, m\}$ is the chosen site in $V$.
Consider the case $\sum_{v \notin J_\eta }\sigma(v) \leq \sum_{v \notin J_\eta }\sigma'(v)$.
If \[U< \frac{1+\tanh\left(\beta \sum_{v \notin J_\eta }\sigma(v) \right)}{2},\] then update the chosen site $I$ of $V$ by +1 and $f_{\sigma,\sigma'}(I)$ of $V'$ by +1.
If \[ U\geq \frac{1+\tanh\left(\beta \sum_{v \notin J_\eta }\sigma'(v) \right)}{2},\] then update the chosen site $I$ of $V$ by -1 and $f_{\sigma,\sigma'}(I)$ of $V'$ by -1.
Otherwise, if \[\frac{1+\tanh\left(\beta \sum_{v \notin J_\eta }\sigma(v) \right)}{2} \leq U < \frac{1+\tanh\left(\beta \sum_{v \notin J_\eta }\sigma'(v) \right)}{2},\] then update the chosen site $I$ of $V$ by -1 and $f_{\sigma,\sigma'}(I)$ of $V'$ by +1.
The other case $\sum_{v \notin J_\eta }\sigma(v) > \sum_{v \notin J_\eta }\sigma'(v)$ can similarly be updated.

Given the chosen site $I$, we call the above procedure of deciding the updating spin in the two chains a \emph{modified monotone update} with respect to the given modified matching.

Now, fix a modified matching $f_{\sigma,\sigma'}$ of $\sigma$ and $\sigma'$.
Let ${\sigma_{t}}$ and ${\sigma'_{t}}$ be chains starting at $\sigma$ and $\sigma'$, respectively.
Repeating the above procedure independently for each step with respect to $f_{\sigma,\sigma'}$ gives a coupling of the Glauber dynamics.
We call this coupling a \emph{modified monotone coupling} with respect to the given modified matching.
\end{defi}
\begin{rmk}
\thref{monotonecontraction} and its consequences hold with a suitable distance function for a modified coupling with respect to a given modified matching.
\end{rmk}

We first construct a coupling such that the magnetizations agree after $t_n +O(n)$ steps in the next two lemmas.

\begin{lem}[Lemma 2.4, \cite{levin}]\thlabel{supermartingale}
    Let $(W_t)_{t \geq 0 }$ be a non-negative supermartingale with a stopping time $\tau$ satisfying 
    \\\emph{(i)} \ $W_{0} =k$
    \\\emph{(ii)} \ $W_{t+1}-W_{t} \leq B <\infty$
    \\\emph{(iii)} \  $\mathbb{V}\mathrm{ar}(W_{t+1}|\mathcal F_t) >\sigma ^2 >0 \ \text{on the event} \ \{\tau >t\}$.
    Then for $u> \frac{4B^2}{3\sigma^2}$, \[\mathbb P_k(\tau >u) \leq \frac{4k}{\sigma\sqrt{u}}.\]
\end{lem}

\begin{lem}[Magnetization coupling]\thlabel{magcoupling}
Let $\beta< \beta_{cr}$. 
For any configurations $\sigma$ and $\sigma'$, there exists a coupling $({\sigma_{t}} ,{\sigma'_{t}})$ with starting states $(\sigma, \sigma')$ satisfying the following condition.
If $\tau_{mag} \colonequals \min \{t \geq 0: \mathbf S_t = \mathbf S_t'\}$, then for large $\gamma n$, \[\mathbb P_{\sigma,\sigma'}(\tau_{mag} >t_n +\gamma n) \leq \frac{c}{\sqrt{\gamma}}\] where $c>0$ is a constant not depending on $\sigma$, $\sigma'$, or $n$.
\end{lem}
\begin{proof}

Let $({\sigma_{t}}, {\sigma'_{t}})$ be a monotone coupling with starting states $(\sigma,\sigma')$.
Put $Y_{i,t} \colonequals \frac{n}{2}a_i|{S^{(i)}_{t}}-{S'^{(i)}_{t}}|$ for $i=1,\dots,m$ and $Y_{tot,t}\colonequals \sum_{i=1}^m Y_{i,t}$.
Define \[\tau \colonequals \min \{t\geq t_n : \max_{1\leq i\leq m}Y_{i,t}/a_i\leq 1\}.\]
By \thref{normcontraction}, \[\mathbb E_{\sigma, \sigma '} [Y_{tot,t_n}] \leq c \sqrt{n}\] for some $c>0$.

We construct a coupling such that $(Y_{tot,t})_{t_n\leq t< \tau}$ is a positive supermartingale with bounded increments and the conditional probability of not being lazy is bounded away from zero uniformly in time and $n$.

To that end, consider a time $t_n\leq t< \tau$.
Define $K_t\colonequals \bigcup_{i: Y_{i,t}/a_i \leq 1} J_i$, $L_t\colonequals \bigcup_{i: Y_{i,t}/a_i > 1} J_i$, and $L_t'\colonequals \bigcup_{i: Y_{i,t}/a_i > 1} J_i'$.
Note that $L_t \neq \emptyset$ since $t < \tau$.
Choose a site equiprobably over $V=K_t \dot\cup L_t$.
Let $f_t$ be the modified matching of ${\sigma_{t}}$ and ${\sigma'_{t}}$.
If a site in $K_t$ is chosen, then use the modified monotone update with respect to $f_t$ to update $({\sigma_{t}},{\sigma'_{t}})$.
If a site in $L_t$ is chosen, then independently choose another site equiprobably over $L_t'$ (which can be the same site) to update ${\sigma'_{t}}$ independent of ${\sigma_{t}}$.
It is easy to check that the above is a coupling of the Glauber dynamics.

Clearly, $Y_{tot,t}$ has bounded increment with the above coupling.
Let $I$ be a random variable uniformly distributed over $V$ which is independent of $\mathcal F_t$.
Let $E= \{I\in L_t,{\sigma_{t}}(I)=+1, \sigma_{t+1}(I)=-1, \sigma_{t+1}'(f_t(I))=1\}$ and $F= \{I\in L_t,{\sigma_{t}}(I)=-1, \sigma_{t+1}(I)=+1, \sigma_{t+1}'(f_t(I))=-1\}$.
Since $L_t \neq \emptyset$ implies $|L_t|/n\geq p_1$, we obtain that $\mathbb P(Y_{tot,t+1}\neq Y_{tot,t}|\mathcal F_t)$ is bounded below by \begin{align*}
    &\geq \mathbb P(Y_{tot,t+1}\neq Y_{tot,t}, I \in L_t|\mathcal F_t) \geq \mathbb P(E\dot\cup F | \mathcal F_t)
    \\&\geq  \frac{|L_t| +\sum_{i\in L_t}{\sigma_{t}}(i)}{2n}\biggl(\frac{1-\tanh(\beta(1-p_1))}{2}\biggr)^2
    \\&\enspace+ \frac{|L_t| -\sum_{i\in L_t}{\sigma_{t}}(i)}{2n}\biggl(\frac{1-\tanh(\beta(1-p_1))}{2}\biggr)^2
    \\&\geq  p_1\biggl(\frac{1-\tanh(\beta(1-p_1))}{2}\biggr)^2>0.
\end{align*}
Finally, we need to show the supermartingale property.
Consider $Y_{1,t+1}/a_1-Y_{1,t}/a_1$.
Suppose $J_1 \subseteq K_t$.
Then by a direct calculation, on the event $\{J_1 \subseteq K_t\}$, it holds that $\mathbb E(Y_{1,t+1}/a_1-Y_{1,t}/a_1|\mathcal F_t)$ is bounded above by \begin{align*}
    &\leq \biggl(p_1 -\frac{|{S^{(1)}_{t}}-{S'^{(1)}_{t}}|}{2}\biggr)\frac{|\tanh(\beta \sum_{j\neq1}{S^{(j)}_{t}})-\tanh(\beta \sum_{j\neq1}{S'^{(j)}_{t}})|}{2}
    \\&\enspace -\frac{|{S^{(1)}_{t}}-{S'^{(1)}_{t}}|}{2}\biggl(1-\frac{|\tanh(\beta \sum_{j\neq1}{S^{(j)}_{t}})-\tanh(\beta \sum_{j\neq1}{S'^{(j)}_{t}})|}{2}\biggr)
    \\&\leq \frac{1}{2}\biggl(-|{S^{(1)}_{t}}-{S'^{(1)}_{t}}|+p_1\tanh\biggl(\beta \Bigl|\sum_{j\neq1}{S^{(j)}_{t}}-\sum_{j\neq1}{S'^{(j)}_{t}}\Bigl|\biggr)\biggr).
\end{align*}
Suppose $J_1 \subseteq L_t$.
Note that $Y_{1,t} >1$ implies $({S^{(1)}_{t+1}}-{S'^{(1)}_{t+1}})({S^{(1)}_{t}}-{S'^{(1)}_{t}})\geq 0 $ and $|{S^{(1)}_{t}}-{S'^{(1)}_{t}}|>0$.
Let $\xi= ({S^{(1)}_{t}}-{S'^{(1)}_{t}})/|{S^{(1)}_{t}}-{S'^{(1)}_{t}}| \in \{\pm 1\}$.
Then by equation \eqref{dynamics} in Section \ref{subsection4.2}, on the event $\{J_1 \subseteq L_t\}$, $\mathbb E(Y_{1,t+1}/a_1-Y_{1,t}/a_1|\mathcal F_t)$ is equal to \begin{align*}
    &= \xi \frac{n}{2}\biggl(\mathbb E({S^{(1)}_{t+1}}-{S^{(1)}_{t}}| {\sigma_{t}})-\mathbb E({S'^{(1)}_{t+1}}-{S'^{(1)}_{t}}| {\sigma'_{t}})\biggr)
    \\&=\xi \frac{n}{2}\frac{1}{n}\biggl(-{S^{(1)}_{t}}+p_1\tanh(\beta \sum_{j\neq1}{S^{(j)}_{t}}) \biggr) 
    \\&\quad - \xi \frac{n}{2}\frac{1}{n}\biggl(-{S'^{(1)}_{t}}+p_1\tanh(\beta \sum_{j\neq1}{S'^{(j)}_{t}}) \biggr)
    \\&= \frac{\xi}{2}\biggl(-({S^{(1)}_{t}}-{S'^{(1)}_{t}})+p_1\biggl(\tanh(\beta \sum_{j\neq1}{S^{(j)}_{t}})-\tanh(\beta\sum_{j\neq1}{S'^{(j)}_{t}})\biggr) \biggr)
    \\&\leq \frac{1}{2}\biggl(-|{S^{(1)}_{t}}-{S'^{(1)}_{t}}|+p_1\tanh\biggl(\beta \Bigl|\sum_{j\neq1}{S^{(j)}_{t}}-\sum_{j\neq1}{S'^{(j)}_{t}}\Bigl|\biggr) \biggr).
\end{align*}
Since either $J_1 \subseteq L_t$ or $J_1 \subseteq K_t$ must hold, $\mathbb E(Y_{1,t+1}/a_1-Y_{1,t}/a_1|\mathcal F_t)$ is equal to \begin{align*}
    &= \mathbbm 1_{J_1 \subseteq K_t}\mathbb E(Y_{1,t+1}-Y_{1,t}|\mathcal F_t) +\mathbbm 1_{J_1 \subseteq L_t}\mathbb E(Y_{1,t+1}-Y_{1,t}|\mathcal F_t)  
    \\&\leq \frac{1}{2}\biggl(-|{S^{(1)}_{t}}-{S'^{(1)}_{t}}|+p_1\tanh\biggl(\beta \Bigl|\sum_{j\neq1}{S^{(j)}_{t}}-\sum_{j\neq1}{S'^{(j)}_{t}}\Bigr|\biggr) \biggr)
    \\&\leq\frac{1}{2}\biggl(-|{S^{(1)}_{t}}-{S'^{(1)}_{t}}|+p_1\beta \sum_{j\neq1}\Bigl|{S^{(j)}_{t}}-{S'^{(j)}_{t}}\Bigl| \biggr).
\end{align*}
Thus, \[ \mathbb E(Y_{1,t+1}/a_1|\mathcal F_t) \leq (1-\frac{1}{n})Y_{1,t}/a_1 +\frac{\beta p_1}{n}\sum_{j\neq 1} Y_{j,t}/a_j.\]
Putting in the matrix form with $ \tilde {\mathbf Y}_t \colonequals (Y_{1,t}/a_1,\dots, Y_{m,t}/a_m)^T$, we have  \begin{align*}
    \mathbb E(Y_{tot,t+1}|\mathcal F_t)= \mathbf a^T \mathbb E(\tilde {\mathbf Y}_{t+1}|\mathcal F_t) \leq \mathbf a^T \mathbf A \tilde {\mathbf Y}_t  = g \mathbf a^T \tilde {\mathbf Y}_t = g Y_{tot,t}.
\end{align*}
Since $\beta<\beta_{cr}$ implies $g<1$ by \thref{upsilon}, the supermartingale property is established.

With the above coupling, by \thref{supermartingale}, for large $\gamma n$, \[\mathbb P_{\sigma,\sigma'}(\tau >t_n +\gamma n|\sigma_{t_n},\sigma'_{t_n}) \leq c' \frac{n\|(S^{(1)}_{t_n},\dots, S^{(m)}_{t_n})-(S'^{(1)}_{t_n},\dots,S'^{(m)}_{t_n})\|_1}{\sqrt{\gamma n}}\] for some $c' >0$ not depending on $n$.
Taking expectation, \[\mathbb P_{\sigma,\sigma'}(\tau >t_n +\gamma n) \leq O(\gamma ^{-1/2}).\]
Note $\sigma_{\tau}$ has at most $m$ more +1 spin sites than $\sigma_{\tau}'$, so $0\leq Y_{tot,\tau} \leq a_1 m$ by \thref{finalfinallemma}.
At $\tau$, construct a modified matching of $\sigma_{\tau}$ and $\sigma_{\tau}'$, and use the modified monotone coupling with respect to this modified matching from then on.
At $\tau_{mag}$, we construct another modified matching of the sites to do a new modified monotone coupling so that $({S^{(1)}_{t}},\dots, {S^{(m)}_{t}})=({S'^{(1)}_{t}},\dots,{S'^{(m)}_{t}})$ forever after $\tau_{mag}$.

By \thref{finalfinallemma}, a modified version of \thref{normcontraction}, and the strong Markov property, we have \begin{align*}
    \mathbb P_{\sigma,\sigma'}(\tau_{mag}>\tau +\gamma' n|\sigma_{\tau}, \sigma_{\tau}') &\leq \mathbb P_{\sigma,\sigma'}(Y_{tot,\tau +\gamma' n}\geq a_m |\sigma_{\tau}, \sigma_{\tau}')
    \\& \leq \mathbb E_{\sigma,\sigma'}[Y_{tot,\tau +\gamma' n}|\sigma_{\tau}, \sigma_{\tau}'] /a_m
    \\& \leq g^{\gamma'n} Y_{tot,\tau}/a_m \leq g^{\gamma'n} a_1m/a_m \leq e^{-\upsilon \gamma'}a_1m/a_m.
\end{align*}
Thus, \begin{align*}
\mathbb P_{\sigma,\sigma'}(\tau_{mag} > t_n + (\gamma+\gamma')n) &\leq O(\gamma^{-1/2})+ e^{-\upsilon \gamma'}a_1m/a_m,
\end{align*} and putting $\gamma=\gamma'$ yields \begin{align*}
\mathbb P_{\sigma,\sigma'}(\tau_{mag} > t_n + \gamma n) \leq O(\gamma^{-1/2}).
\end{align*}
\end{proof}

\begin{defi}[Good configurations]
Define the set of "good" configurations by \[\tilde \Omega \colonequals \{\sigma \in \Omega : |S^{(i)}(\sigma )|\leq p_i/2, \ i=1,\dots,m\}.\]
For $\sigma=(\sigma^{(1)}, \dots,\sigma^{(m)}) \in \tilde \Omega$ and each $i$, define \begin{align*}
&u_{i }^{\sigma} \colonequals |\{v\in J_i: \sigma^{(i)}(v)=1\}|, \enspace v_{i  }^{\sigma} \colonequals |\{v\in J_i: \sigma^{(i)}(v)=-1\}|.
\end{align*}
Define \[\tilde \Lambda \colonequals \{(u_1,v_1,u_2,v_2,\dots,u_m,v_m) \in \mathbb N^{2m} : {|J_i|}/{4} \leq u_i \wedge v_i,\ i=1,\dots,m\}.\]
\end{defi}
\begin{rmk}
Note that $\sigma \in \tilde\Omega \iff (u_1^{\sigma},v_1^{\sigma},\dots,u_m^{\sigma},v_m^{\sigma} ) \in \tilde \Lambda$.
In other words, $\tilde \Lambda$ is another representation of good configurations $\tilde \Omega$. 
We omit the starting state and write $u_i$ instead of $u_i^\sigma$ for convenience.
\end{rmk}

\begin{lem}[Lemma 3.3, \cite{levin}]\thlabel{goodstartingstatesdistance}
For any subset $A\subseteq \Omega$ and stationary distribution $\pi$,
\begin{align*}
    d_n(t_0+t)&=\max_{\sigma \in \Omega} \|\mathbb P_{\sigma }(\sigma_{t_0+t} \in \cdot) -\pi\|_{TV} 
    \\&\leq \max_{\sigma \in A} \|\mathbb P_{\sigma }({\sigma_{t}} \in \cdot) -\pi\|_{TV} +\max_{\sigma \in \Omega} \mathbb P_{\sigma}(\sigma_{t_0} \notin A).
\end{align*}
\end{lem}

Recall that we are assuming the high temperature regime.
By \thref{expectedmagbound}, there exists $\delta >0$ such that $\max_{\sigma \in \Omega,1\leq i\leq m}|\mathbb E_{\sigma} S^{(i)}_{\delta n}| \leq  p_1 /4$.
Hence, by \thref{magvariancebound}, for large $n$,
\begin{align*} \mathbb P_{\sigma}(\sigma_{\delta n}\notin \tilde \Omega ) &\leq \sum_{ i=1}^m \mathbb P_{\sigma}(|S^{(i)}_{\delta n} | >p_i/2) \leq \sum_{i =1}^m \mathbb P_{\sigma}(|S^{(i)}_{\delta n}-\mathbb E_{\sigma}S^{(i)}_{\delta n} | >p_i/4)
\\&\leq \frac{16}{p_1 ^2} \sum_{i=1}^m \mathbb{V}\mathrm{ar}_{\sigma}S^{(i)}_{\delta n}= O(1/n).\end{align*}
Combining with \thref{goodstartingstatesdistance},
\begin{equation}
    d_n(\delta n+t) \leq \max_{\sigma \in \tilde \Omega} \|\mathbb P_{\sigma }({\sigma_{t}} \in \cdot) -\mu\|_{TV} +O(1/n). \label{distance}
\end{equation}

\begin{defi}[$2m$-coordinate chain]
Let $\tilde \sigma \in \Omega$ be a reference configuration.
For $\sigma \in \Omega$ and each $i$, define \begin{align*}
    &U_i(\sigma) \colonequals |\{v\in J_i: \sigma^{(i)}(v)=\tilde \sigma^{(i)}(v)=1\}|, 
    \\& V_i(\sigma) \colonequals |\{v\in J_i: \sigma^{(i)}(v)=\tilde \sigma^{(i)}(v)=-1\}|.
    \end{align*}
For a chain $({\sigma_{t}})$ with the starting configuration $ \sigma_0 \in \Omega$, define the \emph{$2m$-coordinate chain with respect to $\tilde \sigma$} by \begin{align*}
    \mathbf U_t \colonequals ({U^{(1)}_{t}},{V^{(1)}_{t}},\dots,{U^{(m)}_{t}},{V^{(m)}_{t}})\colonequals (U_{1}({\sigma_{t}}),V_{1}({\sigma_{t}}),\dots,U_{m}({\sigma_{t}}),V_{m}({\sigma_{t}})).
\end{align*}
It is easy to see that the $2m$-coordinate chain is again a Markov chain in its state space $\mathcal U \subseteq \mathbb N^{2m}$ and determines the magnetization chain $({S^{(1)}_{t}},\dots,{S^{(m)}_{t}})$ through the relation ${S^{(i)}_{t}}=2({U^{(i)}_{t}}-V^{(i)}_t)/n-( \tilde u_{i}- \tilde v_{i})/{n}$ for $i=1,\dots,m$.
\end{defi}

Symmetry gives us the following lemma which is an adaptation of Lemma 3.4 in \cite{levin}.
\begin{lem}\thlabel{totalvariationdistance}
Let $({\sigma_{t}})$ be a chain starting at $\sigma \in \Omega$.
Consider the corresponding $2m$-coordinate chain starting at $\mathbf u \in \mathcal U$.
Then \[\|\mathbb P_{\sigma}({\sigma_{t}} \in \cdot) - \mu\|_{TV}= \|\mathbb P_{\mathbf u  }(({U^{(1)}_{t}},{V^{(1)}_{t}},\dots, {U^{(m)}_{t}},{V^{(m)}_{t}})\in \cdot )-\nu\|_{TV}\]
where $\nu$ is the stationary distribution of the $2m$-coordinate chain.
\end{lem}
\begin{proof}
Since $\mu(\sigma)=e^{\beta n\sum_{i\neq j}S^{(i)}(\sigma) S^{(j)}(\sigma)}/Z(\beta)$, given the $2m$-coordinate $\mathbf u' \in \mathcal U$, the conditional $\mu$-probability of the configurations is equiprobable.
In other words, $\mu(\cdot | \Omega(\mathbf u'))$ is uniform where $\Omega(\mathbf u')$ is the set of configurations having the $2m$-coordinate $\mathbf u'$.
Also, by symmetry, \[\mathbb P_{\sigma}({\sigma_{t}} \in \cdot \ |\mathbf U_t = \mathbf u')\] is uniform over $\Omega(\mathbf u')$.
Thus, \begin{align*}
&\mathbb P_\sigma ({\sigma_{t}} = \eta) -\mu(\eta) 
=\sum_{\mathbf u' \in \mathcal U}\frac{\mathbbm 1\{\eta \in \Omega(\mathbf u')\}}{|\Omega(\mathbf u')|}\left(\mathbb P_{\mathbf u'}\left(\mathbf U_t = \mathbf u'\right)-\mu (\Omega(\mathbf u'))\right).
\end{align*}
Taking absolute values, applying the triangular inequality, summing over $\eta$, and changing the order of summation shows \[\|\mathbb P_{\sigma}({\sigma_{t}} \in \cdot) - \mu\|_{TV}\leq  \|\mathbb P_{\mathbf u  }(({U^{(1)}_{t}},{V^{(1)}_{t}},\dots, {U^{(m)}_{t}},{V^{(m)}_{t}})\in \cdot )-\nu\|_{TV}.\]
The reverse inequality holds since the $2m$-coordinate chain is a function of the original chain $({\sigma_{t}})$.
\end{proof}
\begin{rmk}
This lemma lets us look at the $2m$-coordinate chain instead of the original chain when considering the total variation distance.
\end{rmk}

Fix a good configuration $\tilde \sigma \in \tilde \Omega$.
Recall $\tau_{mag}$ defined in \thref{magcoupling}.
We use the following coupling after $\tau_{mag}$, which is a generalization of Lemma 3.5 of \cite{levin}.
\begin{lem}[Post magnetization coupling]\thlabel{postmagcoupling}
Let $\tilde \sigma \in \tilde \Omega$ be a good configuration. 
Suppose that two configurations $\sigma_0, \sigma_0'$ satisfy $S^{(i)}(\sigma_0)=S^{(i)}(\sigma_0')$ for $i=1,\dots,m$.
With respect to the good configuration $\tilde \sigma$, define \begin{align*}
\Theta_i \colonequals \Big\{\sigma \in \Omega : \min \{U _i(\sigma), \tilde u_i-U  _i(\sigma), V_i(\sigma),  \tilde v_i-V_i(\sigma)\} \geq \frac{|J_i|}{16}\Big\},\enspace \Theta \colonequals \bigcap _{ i=1}^m\Theta_i \end{align*} for each $i$.
Then there exists a coupling $({\sigma_{t}}, {\sigma'_{t}})$ of the Glauber dynamics with starting states $(\sigma_0,\sigma_0')$ satisfying: \begin{align*}
    &\mathrm {(i)} \; \mathbf S_{t}=\mathbf S_{t}'\ \text{for all}\ t\geq0
    \\&\mathrm {(ii)} \; \text{If ${R^{(i)}_{t}}\colonequals {U'^{(i)}_{t}}-{U^{(i)}_{t}}$, then }
    \mathbb E_{\sigma_0,\sigma_0'}\left({R^{(i)}_{t+1}}-{R^{(i)}_{t}}| {\sigma_{t}}, {\sigma'_{t}} \right) = \frac{-{R^{(i)}_{t}}}{n}, \\&\qquad i=1,\dots,m
    \\&\mathrm {(iii)} \;  \text{There exists $c>0$ not depending on $n$ such that on the event}\  \{{\sigma_{t}},{\sigma'_{t}} \in \Theta\},\\& \qquad \mathbb P_{\sigma_0,\sigma_0'}\left({R^{(i)}_{t+1}}-{R^{(i)}_{t}} \neq 0| {\sigma_{t}}, {\sigma'_{t}} \right) \geq c>0 \; \text{for all} \ i=1,\dots,m.
\end{align*}
\end{lem}
\begin{proof}
We inductively define the coupling.
The random spin $S$ determined by the randomness $I$ and $U$ is \begin{align*}
    S=\sum_{i=1}^m (\mathbbm 1_{I\in J_i,\; U\leq r_+(\sum_{j\neq i}{S^{(j)}_{t}})}-\mathbbm1_{I\in J_i,\; U > r_+(\sum_{j\neq i}{S^{(j)}_{t}})}).
\end{align*}
Suppose that $({\sigma_{t}}, {\sigma'_{t}})$ is given such that the statements hold for some $t\geq0$.
Let $\sigma_{t+1}$ be determined $I$ and $U$.
If $I\in J_i$ for some $i$, then choose $I'$ randomly from $\{v \in J_i' : {\sigma'_{t}}(v) = {\sigma_{t}}(I)\}$.
Update the primed chain by
\[\sigma_{t+1}'(v)=  \begin{cases}  {\sigma'_{t}}(v) & \text{if $v \neq I'$} \\ S  & \text{if $v = I'$} \end{cases} \quad . \]
By the induction hypothesis $\mathbf S_t=\mathbf S_t'$, we have $\{v \in J_i' : {\sigma'_{t}}(v) = {\sigma_{t}}(I)\} \neq \emptyset$ and $({\sigma'_{t}})$ satisfies the Glauber dynamics.
Also, $\mathbf S_{t+1}=\mathbf S_{t+1}'$ with this coupling.

For $i=1,\dots,m$, put \begin{align*}
    &A_i(\sigma) \colonequals \{v \in J_i : \sigma(v)=\tilde \sigma(v)=1\},
    \\&B_i(\sigma) \colonequals \{v \in J_i : \sigma(v)=-1, \ \tilde\sigma(v)=1\},
    \\&C_i(\sigma) \colonequals \{v \in J_i : \sigma(v)=1, \ \tilde\sigma(v)=-1\},
    \\&D_i(\sigma) \colonequals \{v \in J_i : \sigma(v)=\tilde\sigma(v)=-1\},
\end{align*}
so $|A_i(\sigma )|=U_i(\sigma)$, $|B_i(\sigma )|=\tilde u_i-U_i(\sigma)$, $|C_i(\sigma )|=\tilde v_i- V_i(\sigma)$, and $|D_i(\sigma )|=V_i(\sigma)$.

Now we calculate ${R^{(1)}_{t+1}}-{R^{(i)}_{t}}$ with the above coupling.
The following table shows the one-step dynamics of ${R^{(1)}_{t}}$.
\begin{center}
 \begin{tabular}{|c| c| c|c|} 
 \hline
 $I$ & $I'$ & $S$ & ${R^{(1)}_{t+1}}-{R^{(1)}_{t}}$ \\  
 \hline\hline
 $B_{1}({\sigma_{t}})$ & $D_{1}({\sigma'_{t}})$ & 1 & -1  \\ 
 \hline
 $C_{1}({\sigma_{t}})$ & $A_{1}({\sigma'_{t}})$ & -1 & -1 \\
 \hline
 $A_{1}({\sigma_{t}})$ & $C_{1}({\sigma'_{t}})$ & -1 & 1  \\
 \hline
 $D_{1}({\sigma_{t}})$ & $B_{1}({\sigma'_{t}})$ & 1 & 1  \\
 \hline
 otherwise &otherwise&otherwise&0\\
 \hline
\end{tabular}
\end{center}
Since ${S^{(1)}_{t}}={S'^{(1)}_{t}}$ implies ${R^{(1)}_{t}}\equiv {U'^{(1)}_{t}}-{U^{(1)}_{t}}={V'^{(1)}_{t}}-{V^{(1)}_{t}}$, \begin{align*}
    &\mathbb P_{\sigma_0,\sigma_0'}({R^{(1)}_{t+1}}-{R^{(1)}_{t}}=-1|{\sigma_{t}}, {\sigma'_{t}})\equalscolon a({U^{(1)}_{t}},{V^{(1)}_{t}},U_{2,t},V_{2,t})
    \\&= \frac{\tilde u_1 -{U^{(1)}_{t}}}{n}\frac{{V'^{(1)}_{t}}}{\tilde u_1-{U'^{(1)}_{t}}+{V'^{(1)}_{t}}}r_+(\sum_{j\neq 1}{S^{(j)}_{t}})
    +\frac{\tilde v_1 -{V^{(1)}_{t}}}{n}\frac{{U'^{(1)}_{t}}}{\tilde v_1-{V'^{(1)}_{t}}+{U'^{(1)}_{t}}}r_-(\sum_{j\neq 1}{S^{(j)}_{t}})
    \\&=\frac{\tilde u_1-{U^{(1)}_{t}}}{n}\frac{{V^{(1)}_{t}}+{R^{(1)}_{t}}}{\tilde u_1-{U^{(1)}_{t}}+{V^{(1)}_{t}}}r_+(\sum_{j\neq 1}{S^{(j)}_{t}})
    +\frac{\tilde v_1-{V^{(1)}_{t}}}{n}\frac{{U^{(1)}_{t}}+{R^{(1)}_{t}}}{\tilde v_1-{V^{(1)}_{t}}+{U^{(1)}_{t}}}r_-(\sum_{j\neq 1}{S^{(j)}_{t}}).
\end{align*}
Likewise, \begin{align*}
    &\mathbb P_{\sigma_0,\sigma_0'}({R^{(1)}_{t+1}}-{R^{(1)}_{t}}=1|{\sigma_{t}}, {\sigma'_{t}})\equalscolon b({U^{(1)}_{t}},{V^{(1)}_{t}},U_{2,t},V_{2,t})
    \\&= \frac{{U^{(1)}_{t}}}{n}\frac{\tilde v_1 - {V'^{(1)}_{t}}}{{U'^{(1)}_{t}}+\tilde v_1-{V'^{(1)}_{t}}}r_-(\sum_{j\neq 1}{S^{(j)}_{t}})
    +\frac{{V^{(1)}_{t}}}{n}\frac{\tilde u_1 - {U'^{(1)}_{t}}}{\tilde u_1-{U'^{(1)}_{t}}+{V'^{(1)}_{t}}}r_+(\sum_{j\neq 1}{S^{(j)}_{t}})
    \\&= \frac{{U^{(1)}_{t}}}{n}\frac{\tilde v_1 - ({V^{(1)}_{t}}+{R^{(1)}_{t}})}{{U^{(1)}_{t}}+\tilde v_1-{V^{(1)}_{t}}}r_-(\sum_{j\neq 1}{S^{(j)}_{t}})
    +\frac{{V^{(1)}_{t}}}{n}\frac{\tilde u_1 - ({U^{(1)}_{t}}+{R^{(1)}_{t}})}{\tilde u_1-{U^{(1)}_{t}}+{V^{(1)}_{t}}}r_+(\sum_{j\neq 1}{S^{(j)}_{t}}).
\end{align*}
Thus, by a direct calculation, \begin{align*}
&\mathbb E_{\sigma_0,\sigma_0'}({R^{(1)}_{t+1}}-{R^{(1)}_{t}}|{\sigma_{t}},{\sigma'_{t}})=b-a
\\&=\frac{-{R^{(1)}_{t}}}{n}\biggl(r_+(\sum_{j\neq 1}{S^{(j)}_{t}})+r_-(\sum_{j\neq 1}{S^{(j)}_{t}})\biggr)= \frac{-{R^{(1)}_{t}}}{n}.
\end{align*}

Moreover, on the event $\{{\sigma_{t}},{\sigma'_{t}} \in \Theta\}$, $(\tilde u_1,\tilde v_1,\dots,\tilde u_m,\tilde v_m) \in \tilde \Lambda$ implies ${U^{(1)}_{t}}\leq \tilde u_1 -|J_1|/16 \leq 3|J_1|/4 -|J_1|/16 = 11|J_1|/16$, and $\tilde u_1 -{U^{(1)}_{t}} \leq 3|J_1|/4 -|J_1|/16 = 11|J_1|/16$.
The same upper bound holds for $\tilde v_1-{V^{(1)}_{t}}$ and ${V^{(1)}_{t}}$.
Thus, on the event $\{{\sigma_{t}},{\sigma'_{t}} \in \Theta\}$, \begin{align*}
\mathbb P_{\sigma_0,\sigma_0'}({R^{(1)}_{t+1}}-{R^{(1)}_{t}} \neq 0|{\sigma_{t}},{\sigma'_{t}}) &\geq b\geq \frac{p_1}{16}\frac{\frac{1}{16}r_-(\sum_{j\neq 1}{S^{(j)}_{t}})}{\frac{11}{16}+\frac{11}{16}}+\frac{p_1}{16}\frac{\frac{1}{16}r_+(\sum_{j\neq 1}{S^{(j)}_{t}})}{\frac{11}{16}+\frac{11}{16}}
\\ &=\frac{p_1}{352}. \end{align*}
Similarly, for $i>1$, $\mathbb P_{\sigma_0,\sigma_0'}({R^{(i)}_{t+1}}-{R^{(i)}_{t}} \neq 0|{\sigma_{t}},{\sigma'_{t}}) \geq {p_i}/{352} \geq {p_1}/{352}>0$, which concludes the induction.
\end{proof}

\section{Upper and Lower Bounds in the high temperature regime}\label{section5}
\subsection{Upper Bound}
\begin{thm}\thlabel{upperbound}
For $\beta<\beta_{cr}$, we have \[\lim_{\gamma \to \infty }\limsup_{n \to \infty}d_n(t_n +\gamma n)=0.\]
\end{thm}
\begin{proof}
Let $\nu$ be the stationary measure for the $2m$-coordinate chain.
For any $A\subseteq \mathcal U$, \begin{align*}
    |\mathbb P_{\mathbf u}(\mathbf U_t \in A)-\nu(A)| &= \Big|\sum_{\mathbf u' \in \mathcal U} \nu(\mathbf u') \left(\mathbb P_{\mathbf u}(\mathbf U_t \in A)-\mathbb P_{\mathbf u'}(\mathbf U_t' \in A)\right) \Big| 
    \\ &\leq \sum_{\mathbf u' \in \mathcal U} \nu(\mathbf u')  \|\mathbb P_{\mathbf u}(\mathbf U_t \in \cdot)-\mathbb P_{\mathbf u'}(\mathbf U_t' \in \cdot)\|_{TV}
    \\ &\leq \max_{\mathbf u' \in \mathcal U}\|\mathbb P_{\mathbf u}(\mathbf U_t \in \cdot)-\mathbb P_{\mathbf u'}(\mathbf U_t' \in \cdot)\|_{TV}. 
\end{align*}
Thus, taking supremum over $A\subseteq \mathcal U$ and $\mathbf u\in \tilde \Lambda$,
\begin{align*}
    \max_{\mathbf{u}\in \tilde \Lambda}\|\mathbb P_{\mathbf u}\left(\mathbf U_t\in \cdot \right)-\nu\|_{TV} \leq \max_{\substack{\mathbf{u}\in \tilde \Lambda,\\ \mathbf{u}' \in \mathcal U}}\| \mathbb P_{\mathbf{u}}\left(\mathbf U_t\in \cdot \right)-\mathbb P_{\mathbf{u}'}\left(\mathbf U_t'\in \cdot \right)\|_{TV}.
\end{align*}
Also, from inequality \eqref{distance} and \thref{totalvariationdistance}, \begin{align*}
    d_n(\delta n+t) &\leq \max_{\sigma \in \tilde \Omega} \|\mathbb P_{\sigma }({\sigma_{t}} \in \cdot) -\mu\| +O(1/n) 
    \\&= \max_{\mathbf{u} \in \tilde \Lambda }\|\mathbb P_{\mathbf{u}}(\mathbf U_t\in \cdot )-\nu\|_{TV} +O(1/n).
\end{align*}
For $2m$-coordinate chains $\mathbf U_t$ and $\mathbf U_t'$ with respect to a fixed $\tilde \sigma\in \tilde \Omega$ starting at $\mathbf u \in \mathcal U$ and $\mathbf u'\in \mathcal U$, respectively, put \[\tau_{tot,c} \colonequals \min\{t \geq 0 :\mathbf U_t= \mathbf U_t'\}.\]
It is a standard fact \parencite[Section 5.2]{Peres} that \begin{align*}
    \|\mathbb P_{\mathbf u}(\mathbf U_t\in \cdot )-\mathbb P_{\mathbf u'}(\mathbf U_t'\in \cdot )\|_{TV} 
     \leq \mathbb P_{\mathbf u,\mathbf u'}(\tau_{tot,c} >t).
\end{align*}
Combining all the above results, it suffices to bound \[
\max_{\substack{\mathbf{u}\in \tilde \Lambda,\\ \mathbf{u}' \in \mathcal U}}\mathbb P_{\mathbf u,\mathbf u'}(\tau_{tot,c} >t).\]

With the above considerations, fix a good starting configuration $\tilde \sigma \in \tilde \Omega$ with the associated $2m$-coordinates $ \tilde {\mathbf  u}=(\tilde u_1,\tilde v_1,\dots, \tilde u_m,\tilde v_m) \in \tilde \Lambda$ and an arbitrary starting configuration $\sigma' \in \Omega$.
Put \[t_n(\gamma) \colonequals t_n +\gamma n , \enspace H_M \colonequals \{\tau_{mag}\leq t_n(\gamma)\}.\]

The first step is the magnetization coupling phase.
By \thref{magcoupling}, there exists a coupling $({\sigma_{t}}, {\sigma'_{t}})$ for $t\leq t_n(\gamma)$ with starting configurations $(\tilde \sigma, \sigma')$ such that \[\mathbb P_{\tilde \sigma,\sigma'}(H_M^c) \leq O(1/ \sqrt{\gamma}).\]
The next step is the $2m$-coordinate chain coupling phase.
For $i=1,\dots,m$, define \begin{align*}
    &\tau_{i,c} \colonequals \min\{t \geq 0 : ({U^{(i)}_{t}},V^{(i)}_t)=({U'^{(i)}_{t}},V'^{(i)}_t)\},
    \\ &\Theta_i \colonequals \Big\{\sigma \in \Omega : \min \{U _i(\sigma), \tilde u_i-U  _i(\sigma), V_i(\sigma),  \tilde v_i-V_i(\sigma)\} \geq \frac{|J_i|}{16}\Big\},
    \\ &H_{i}(t) \colonequals \{{\sigma^{(i)}_t}, {\sigma'^{(i)}_t} \in \Theta_i \},
    \quad H_i \colonequals \bigcap_{t\in [t_n(\gamma), t_n(2\gamma)]}H_i (t), \quad H_{tot} \colonequals \bigcap_{i=1}^mH_i.
\end{align*}
We have defined the two coordinate chains with respect to $\tilde \sigma$.
On the event $H_M$, for $t \geq t_n(\gamma)$, we use the coupling in \thref{postmagcoupling}, while on the event $H_M ^c$, we let the chains run independently for $t \geq t_n(\gamma)$ since we do not care about this un-probable event.

Our first claim is that
\begin{align*}
    \mathbb P_{\tilde\sigma,\sigma'}(H_i^c) \leq \gamma O(1/n), \enspace i=1,\dots,m.
\end{align*}
To that end, observe that    
\begin{equation*}
    \begin{split}
        \{{\sigma^{(i)}_t}\notin \Theta_i\} \subseteq  \{{U^{(i)}_{t}}<|J_i|/16\} &\cup \{\tilde u_i-{U^{(i)}_{t}}<|J_i|/16\} 
        \\& \cup \{V^{(i)}_t<|J_i|/16\}\cup \{\tilde v_i-V^{(i)}_t<|J_i|/16\}.
    \end{split}
\end{equation*}
Notice $\tilde u_i \geq |J_i|/4$ implies 
\begin{align*}
    \{{U^{(i)}_{t}}<|J_i|/16\}&\subseteq \{\tilde u_i-{U^{(i)}_{t}}> 3|J_i|/16\},
    \\ \{\tilde u_i-{U^{(i)}_{t}}<|J_i|/16\}&\subseteq\{{U^{(i)}_{t}}> 3|J_i|/16\}.
\end{align*}
Similarly, $\tilde v_i \geq |J_i|/4$ implies
\begin{equation*}
    \begin{split}
        \{V^{(i)}_t<|J_i|/16\}&\subseteq \{\tilde v_i-V^{(i)}_t> 3|J_i|/16\},
        \\ \{\tilde v_i-V^{(i)}_t<|J_i|/16\}&\subseteq \{V^{(i)}_t> 3|J_i|/16\}.
    \end{split}
\end{equation*}

Put \[\tilde A_{i}\colonequals \{k\in J_i : \tilde \sigma (k)=1\},\enspace i=1,\dots,m.\]
Then, following the notation in \thref{expectedmagbound}, $|M_t(\tilde A_i)|=|{U^{(i)}_{t}}-(\tilde u_i - {U^{(i)}_{t}})|$ implies \[\{{U^{(i)}_{t}}<|J_i|/16\}\cup \{\tilde u_i-{U^{(i)}_{t}}<|J_i|/16\} \subseteq \{|M_t(\tilde A_i)|\geq |J_i|/8\}.\]
Similarly, $|M_t(J_i\setminus \tilde A_i)|=|V^{(i)}_t-(\tilde v_i - V^{(i)}_t)|$ implies \[ \{V^{(i)}_t<|J_i|/16\}\cup\{\tilde v_i-V^{(i)}_t<|J_i|/16\}\subseteq \{|M_t(J_i\setminus \tilde A_i)|\geq |J_i|/8\}.\] 
Combining all the above results, we obtain \[\{{\sigma^{(i)}_t} \notin \Theta_i\} \subseteq \{|M_t(\tilde A_i)|\geq |J_i|/8 \} \cup \{|M_t(J_i\setminus \tilde A_i)|\geq |J_i|/8\}.\]
A parallel argument for the primed chain shows \[\{{\sigma'^{(i)}_t} \notin \Theta_i\} \subseteq \{|M_t'(\tilde A_i)|\geq |J_i|/8 \} \cup \{|M_t'(J_i\setminus \tilde A_i)|\geq |J_i|/8\}.\]
In conclusion, 
\begin{align*}
    H_{i}(t) ^c&=\{{\sigma^{(i)}_t} \notin \Theta_i\} \cup \{{\sigma'^{(i)}_t} \notin \Theta_i\}
    \\ & \subseteq \{|M_t(\tilde A_i)|\geq |J_i|/8 \} \cup \{|M_t(J_i\setminus \tilde A_i)|\geq |J_i|/8\} 
    \\  & \qquad \qquad  \qquad \qquad \quad \enspace \; \cup \{|M_t'(\tilde A_i)|\geq |J_i|/8 \} \cup \{|M_t'(J_i\setminus \tilde A_i)|\geq |J_i|/8\}.
\end{align*}

Define \begin{equation*}
    B\colonequals \bigcup_{t\in [t_n(\gamma), t_n(2\gamma)]} \{|M_t(\tilde A_i)|\geq |J_i|/8\} , \quad Y \colonequals \sum_{t\in [t_n(\gamma), t_n(2\gamma)]} \mathbbm 1_{\{ | M_t(\tilde A_i)| \geq |J_i|/16 \}}.
\end{equation*}
Since $M_t(\tilde A_i)$ has increments in $\{-1,0,1\}$, we have $B\subseteq \{Y\geq |J_i|/16\}$.
By Chebyshev's inequality, $\mathbb P_{\tilde\sigma,\sigma'}(B)\leq c\mathbb E_{\tilde\sigma,\sigma'}(Y)/n$ for some constant $c>0$.
From \thref{expectedmagbound}, for $t\geq t_n$, $\mathbb P_{\tilde\sigma,\sigma'}( | M_t(\tilde A_i)| \geq |J_i|/16 )=O(1/n)$, so $\mathbb E_{\tilde\sigma,\sigma'}(Y)=\gamma O(1)$.
Thus, $\mathbb P_{\tilde\sigma,\sigma'}(B)=\gamma O(1/n)$.
Similar results hold for $\bigcup_{t\in [t_n(\gamma), t_n(2\gamma)]} \{|M_t(J_i\setminus \tilde A_i)|\geq |J_i|/8\}$, $\bigcup_{t\in [t_n(\gamma), t_n(2\gamma)]} \{|M_t'( \tilde A_i)|\geq |J_i|/8\}$, and $\bigcup_{t\in [t_n(\gamma), t_n(2\gamma)]} \{|M_t'(J_i\setminus \tilde A_i)|\geq |J_i|/8\}$.
In conclusion, \begin{align*}
    \mathbb P_{\tilde\sigma,\sigma'}(H_i^c)=\mathbb P_{\tilde\sigma,\sigma'}\Biggl(\bigcup_{t\in [t_n(\gamma), t_n(2\gamma)]} H_i(t)^c\Biggr) \leq  4\gamma O(1/n),
\end{align*}
which proves our first claim.

From the first claim, \[\mathbb P_{\tilde \sigma,\sigma'}(H_{tot}^c) \leq \sum_{i=1}^m \mathbb P_{\sigma,\sigma'}(H_{i}^c)= \gamma O(1/n).\]

Now, condition on the event $H_M$.
Recalling the fact that \thref{postmagcoupling} assures $\mathbf S_t= \mathbf S_t'$ for $t \geq t_n(\gamma)$ on the event $H_M$, we can make ${R^{(i)}_{t}}$ stay zero after $\tau_{i,c}$, using the modified monotone update on $J_i$ whenever a site in $J_i$ is chosen to be updated.
Thus, on $H_M$, \[\tau_{tot,c}=\max_{1\leq i\leq m} \tau_{i,c}.\]

Our second claim is that  \begin{align*}
    \mathbb P_{\tilde \sigma,\sigma'}(\tau_{i,c} > t_n(2\gamma) , H_i , H_M)= O(1/\sqrt{\gamma}), \enspace i=1,\dots,m.
\end{align*}
From \thref{supermartingale} and \thref{postmagcoupling}, $\mathbb P_{\tilde \sigma,\sigma'}(\tau_{i,c} > t_n(2\gamma) , H_i , H_M|\sigma_{t_n(\gamma)},\sigma_{t_n(\gamma)}')\leq  {c|R^{(i)}_{t_n (\gamma)}|}/{\sqrt{n\gamma}}$ for some $c>0$.
Taking expectation yields, \[\mathbb P_{\tilde \sigma,\sigma'}(\tau_{i,c} > t_n(2\gamma) , H_i , H_M)\leq  \frac{c\mathbb E_{\tilde \sigma,\sigma'}|R^{(i)}_{t_n (\gamma)}|}{\sqrt{n\gamma}}\]
However, for any $t>0$, $|{R^{(i)}_{t}}|=|U_t'-U_t|=|M_t'(\tilde A_i)-M_t(\tilde A_i)|$, so from \thref{expectedmagbound}, $\mathbb E_{\tilde \sigma,\sigma'}|R^{(i)}_{t_n (\gamma)}| \leq \mathbb E_{\sigma'}|M_{t_n(\gamma)}'(\tilde A_i)|+\mathbb E_{\tilde \sigma}|M_{t_n(\gamma)}(\tilde A_i)| =O(\sqrt{n})$, which proves our second claim.

From the second claim,   \begin{align*}
    &\mathbb P_{\tilde \sigma,\sigma'}(\tau_{tot,c}>t_n( 2\gamma), H_{tot}, H_M) 
    \leq \sum_{i=1}^m \mathbb P_{\tilde \sigma,\sigma'}(\tau_{i,c}>t_n( 2\gamma), H_{tot}, H_M)
    \\&\leq \sum_{i=1}^m \mathbb P_{\tilde \sigma,\sigma'}(\tau_{i,c}>t_n( 2\gamma), H_{i}, H_M)
    = O(1/\sqrt{\gamma}).
\end{align*}
Combining all the above results,   \begin{align*}
    &\mathbb P_{\tilde \sigma,\sigma'}(\tau_{tot,c}>t_n( 2\gamma)) 
    \\&\leq \mathbb P_{\tilde \sigma,\sigma'}(\tau_{tot,c}>t_n( 2\gamma), H_{tot},H_M)+\mathbb P_{\tilde \sigma,\sigma'}(H_{tot}^c) +\mathbb P_{\tilde \sigma,\sigma'}(H_M^c)  
    \\&= O(1/\sqrt{\gamma})+\gamma O(1/n) +O(1/\sqrt{\gamma}).
\end{align*}
Finally, \begin{align*}
    d_n(t_n+(2\gamma+\delta)n) \leq O(1/\sqrt{\gamma}) + \gamma O(1/n)+O(1/n),
\end{align*} which gives us the result upon taking limits.
\end{proof}

\subsection{Lower Bound} \label{subsection4.2}
We first analyze the drift of magnetization chains.
Let $1\leq i\leq m$ and $\mathcal F_t$ be the $\sigma$-algebra generated by ${S^{(1)}_{t}},\dots, {S^{(m)}_{t}}$.
By a direct calculation,

\begin{align}
        \mathbb E[{S^{(i)}_{t+1}}-{S^{(i)}_{t}}|\mathcal F_t]  &= \frac{2}{n}p_i\frac{|J_i|-n{S^{(i)}_{t}}}{2|J_i|}r_+(\sum_{j\neq i}{S^{(j)}_{t}})  -\frac{2}{n}p_i\frac{|J_i|+n{S^{(i)}_{t}}}{2|J_i|}r_-(\sum_{j\neq i}{S^{(j)}_{t}}) \nonumber
        \\ &=\frac{2}{n}\frac{p_i-{S^{(i)}_{t}}}{2}r_+(\sum_{j\neq i}{S^{(j)}_{t}}) -\frac{2}{n}\frac{p_i+{S^{(i)}_{t}}}{2}r_-(\sum_{j\neq i}{S^{(j)}_{t}}) \nonumber
        \\ &= \frac{1}{n}\biggl(-{S^{(i)}_{t}}+p_i\tanh ({\beta} \sum_{j\neq i}{S^{(j)}_{t}})\biggr). \label{dynamics}
\end{align}

The following simple lemma is the main tool to get the lower bound in \thref{lowerbound}.
\begin{lem}[Proposition 7.9, \cite{Peres}]\thlabel{statistics}
Let $f\colon\mathcal{S} \to \mathbb R$ be a measurable function and $\nu_1$, $\nu_2$ be two probability measures on $\mathcal S$.
Let $\sigma_*^2 \colonequals \max\{\mathbb{V}\mathrm{ar}_{\nu_1}f, \ \mathbb{V}\mathrm{ar}_{\nu_2}f\}$.
If $|\mathbb E_{\nu_1}f -\mathbb E_{\nu_2}f|\geq r \sigma_*$, then \[\|\nu_1 -\nu_2\|_{TV}\geq 1-\frac{8}{r^2}\]
\end{lem}

Positive starting configurations give us the following result.
\begin{lem}\thlabel{finallemma}
Let $\mathbf s\geq \mathbf 0$ be the starting magentization.
Then, for $t\geq0$, \[  \mathbb E_{\mathbf s}\|\mathbf S_t\|_1 \leq  g^t\left(\sum_{i=1}^m \frac{(s^{(i)})^2}{p_i}\right)^{1/2}+O(1/\sqrt{n}).\]
\end{lem}
\begin{proof}
Consider the case that $|J_i|$ is odd for each $i=1,\dots,m$.
Let $\nu $ be the starting distribution such that $\mathbf s_+'=(\frac{1}{n},\dots,\frac{1}{n})$ with probability $\frac{1}{2}$ and $\mathbf s_-'=(-\frac{1}{n},\dots,-\frac{1}{n})$ with probability $\frac{1}{2}$.

By \thref{generalcontraction}, since $\mathbf s \geq\mathbf s_+'$ in this case, \begin{align*}
     \mathbf 0 &\leq \mathbb E_{\mathbf s,\nu} (\mathbf S_t-\mathbf S_t')
     \leq \frac{1}{2}\mathbf A^t(\mathbf s - \mathbf s_+')+\frac{1}{2}\mathbf A^t(\mathbf s - \mathbf s_-') =\mathbf A^t\mathbf s.
\end{align*}
However, $\mathbb E_{\nu} {S'^{(i)}_{t}}=0$ for $i=1,\dots,m$ by the remark after \thref{magmarkov}. 
Thus, $\mathbf 0 \leq \mathbb E_{\mathbf s} \mathbf S_t \leq \mathbf A^t \mathbf s$, so by \thref{newlemma}, \begin{align*}
    0 \leq \sum_{i=1}^m \mathbb E_{\mathbf s} {S^{(i)}_{t}} \leq \|\mathbf A^t \mathbf s\|_1  \leq  g^t \left(\sum_{i=1}^m \frac{(s^{(i)})^2}{p_i}\right)^{1/2}.
\end{align*}
From \thref{magvariancebound} and Cauchy-Schwartz inequality, since $0\leq \mathbb E_{\mathbf s}{S^{(i)}_{t}}$ for $i=1,\dots,m$, \begin{align*}
     &\mathbb E_{\mathbf s}\|\mathbf S_t\|_1=\sum_{i=1}^m \mathbb E_{\mathbf s}|{S^{(i)}_{t}}| \leq \sum_{i=1}^m\left(|\mathbb E_{\mathbf s}{S^{(i)}_{t}}|+\sqrt{\mathbb{V}\mathrm{ar}_{\mathbf s}{S^{(i)}_{t}}}\right)
    =\sum_{i=1}^m\mathbb E_{\mathbf s}{S^{(i)}_{t}}+\sum_{i=1}^m\sqrt{\mathbb{V}\mathrm{ar}_{\mathbf s}{S^{(i)}_{t}}}
    \\&\leq  g^t \Biggl(\sum_{i=1}^m \frac{(s^{(i)})^2}{p_i}\Biggr)^{1/2}+ \Biggl(m \sum_{i=1}^m\mathbb{V}\mathrm{ar} {S^{(i)}_{t}}\Biggr)^{1/2}
    = g^t \Biggl(\sum_{i=1}^m \frac{(s^{(i)})^2}{p_i}\Biggr)^{1/2} + O(\frac{1}{\sqrt{n}}).
\end{align*}

Other cases of $|J_i|$ can similarly be shown by considering $0$ instead of $\frac{1}{n}$ whenever the partition has even number of sites.
\end{proof}

Finally, we prove the lower bound.
\begin{thm}\thlabel{lowerbound}
For $\beta< \beta_{cr}$, we have \[\lim_{\gamma \to \infty }\liminf_{n \to \infty}d_n(t_n -\gamma n)=1.\]
\end{thm}
\begin{proof}
Since the magnetization chain is a projection of the original chain, it suffices to provide a lower bound on the total variation norm of the magnetization chain.
Using $\tanh x \geq x-x^2/3$ for $x\in\mathbb R$, from equations \eqref{dynamics}, we have \begin{align*}
    \mathbb E({S^{(i)}_{t+1}}|\mathcal F_t) &\geq (1-\frac{1}{n}){S^{(i)}_{t}}+ \frac{p_i}{n}\Biggl(\beta \sum_{j\neq i}{S^{(j)}_{t}} - \frac{\beta^2 (\sum_{j\neq i}{S^{(j)}_{t}})^2}{3}\Biggr)
\end{align*}
for each $i=1,\dots,m$.
In the matrix form, \begin{align*}
    \mathbb E(\mathbf S_{t+1}|\mathcal F_t) &\geq \mathbf A \mathbf S_t - \mathbf x
\end{align*}
where $\mathbf x = \frac{\beta ^2}{3n} (p_1(\sum_{j\neq 1}{S^{(j)}_{t}})^2,\dots, p_m(\sum_{j\neq m}{S^{(j)}_{t}})^2)^T$.
Recall the definition of $\mathbf a^T \colonequals (a_1, \dots, a_m)>\mathbf 0$ with $\|\mathbf a\|_1=1$ being the left eigenvector of $\mathbf A$ with eigenvalue $g$.
Then $\mathbb E(\mathbf a^T \mathbf S_{t+1}|\mathcal F_t )\geq  \mathbf a^T \mathbf A\mathbf S_t - \mathbf a^T \mathbf x= g\mathbf a^T \mathbf S_t - \mathbf a^T \mathbf x$, i.e.,
\begin{equation}\label{compute}
     \mathbb E\Bigl(\sum_{i=1}^m a_i{S^{(i)}_{t+1}}|\mathcal F_t\Bigr) \geq g \sum_{i=1}^m a_i {S^{(i)}_{t}}- \frac{\beta^2}{3n} \sum_{i=1}^m a_ip_i\biggl(\sum_{j\neq i}{S^{(j)}_{t}}\biggr)^2 . 
\end{equation}
Observe that \begin{align*}
    \sum_{i=1}^m a_ip_i\biggl(\sum_{j\neq i}{S^{(j)}_{t}}\biggr)^2 \leq  \sum_{k=1}^m a_kp_k\bigg( \sum_{j=1}^m |S^{(j)}_{t}| \bigg)^2= \biggl(\sum_{k=1}^m a_k p_k\biggr)\|\mathbf S_t\|_1^2.
\end{align*}
Thus, upon taking expectation in equation \eqref{compute}, \begin{align*}
    \mathbb E\left(\sum_{i=1}^m a_i{S^{(i)}_{t+1}}\right) \geq g \mathbb E\left(\sum_{i=1}^m a_i {S^{(i)}_{t}}\right)- \frac{\beta^2}{3n}\left(\sum_{i=1}^m a_ip_i\right)\mathbb E \|\mathbf S_t\|_1^2  .
\end{align*}

We claim that, \begin{align*}
    \mathbb E\|\mathbf S_t\|_1^2 \leq (\mathbb E\|\mathbf S_t\|_1)^2 + O(1/n).
\end{align*}
Since $\mathbb E\|\mathbf S_t\|_1^2 = (\mathbb E\|\mathbf S_t\|_1)^2 + \mathbb{V}\mathrm{ar}\|\mathbf S_t\|_1$, it suffices to show $\mathbb{V}\mathrm{ar}\|\mathbf S_t\|_1 \leq O(1/n)$.
However, from \thref{magvariancebound}, \begin{align*}
 \mathbb{V}\mathrm{ar}\|\mathbf S_t\|_1 &=\sum_{i=1}^m \mathbb{V}\mathrm{ar}|{S^{(i)}_{t}}| + 2\sum_{i> j}\mathrm{Cov}(|{S^{(i)}_{t}}|,|{S^{(j)}_{t}}|)
 \\&\leq \sum_{i=1}^m \mathbb{V}\mathrm{ar}{S^{(i)}_{t}} + 2\sum_{i>j}\sqrt{\mathbb{V}\mathrm{ar}{S^{(i)}_{t}}}\sqrt{\mathbb{V}\mathrm{ar}{S^{(j)}_{t}}}
 \\&\leq \sum_{i=1}^m \mathbb{V}\mathrm{ar}{S^{(i)}_{t}} + \sum_{i>j}(\mathbb{V}\mathrm{ar}{S^{(i)}_{t}}+\mathbb{V}\mathrm{ar}{S^{(j)}_{t}})=m\sum_{i=1}^m\mathbb{V}\mathrm{ar}{S^{(i)}_{t}}=O(1/n), \end{align*}
which proves the claim.

Put $Z_t \colonequals \sum_{i=1}^m a_i {S^{(i)}_{t}}/g^t$.
Then, from the claim above, \begin{align*}
\mathbb EZ_{t+1}-\mathbb EZ_t\geq -\frac{\beta^2 \sum_{i} a_ip_i}{3ng^{t+1}}\left((\mathbb E\|\mathbf S_t\|_1)^2 + O(1/n)\right).
\end{align*}
Assume that $\mathbf s \geq \mathbf 0$ is a non-negative starting magnetization.
Recalling the definition $\upsilon \colonequals n(1-g)$, from \thref{finallemma} and the fact $\sum_{i} {(s^{(i)})^2}/{p_i} \leq 1$,  \begin{align*}
    \mathbb E_{\mathbf s}Z_{t+1}-\mathbb E_{\mathbf s}Z_t &\geq -\frac{\beta^2\sum_{i} a_ip_i}{3ng^{t+1}}\Biggl(\biggl(g^t\biggl(\sum_{i} {(s^{(i)})^2}/{p_i}\biggr)^{1/2}+O(1/\sqrt{n})\biggr)^2+O(1/n)\Biggr)
    \\&\geq -\frac{\beta^2\sum_{i} a_ip_i }{3(n-\upsilon)}\Biggl(g^t\sum_{i} {(s^{(i)})^2}/{p_i}+O(1/\sqrt{n})+\frac{1}{g^t}O(1/n)\Biggr).
\end{align*}
Iterating from $0$ to $t-1$, \begin{align*}
    \mathbb E_{\mathbf s}Z_{t}-Z_0 &\geq -\frac{\beta^2\sum_{i}a_ip_i}{3(n-\upsilon)}\left(\frac{1-g^t}{\upsilon/n}\sum_{i=1}^m \frac{(s^{(i)})^2}{p_i}+tO(1/\sqrt{n})+\frac{n-\upsilon}{\upsilon}(\frac{1}{g^t}-1)O(1/n)\right)
    \\&=-\frac{\beta^2\sum_{i}a_ip_i}{3\upsilon(1-\upsilon/n)}(1-g^t)\sum_{i=1}^m \frac{(s^{(i)})^2}{p_i}-\frac{\beta^2\sum_{i}a_ip_i}{3(n-\upsilon)}tO(1/\sqrt{n})
    \\&\quad\;-\frac{\beta^2\sum_{i}a_ip_i}{3\upsilon}(\frac{1}{g^t}-1)O(1/n).
\end{align*}

For brevity, let us prefer to use $\upsilon$ rather than use $\beta_{cr}$ in view of \thref{upsilon}.
Consider the step $t_*\colonequals t_n-\gamma n/\upsilon=\frac{1}{2\upsilon} n\ln n- \frac{\gamma n}{\upsilon}$.
Observe that $1-1/x \geq e^{-1/(x-1)}$ for $x>1$ implies \[g^{t_*} \geq \frac{e^{\gamma}}{n^{n/(2(n-\upsilon))}}.\]
Then \begin{align*}
    \mathbb E_{\mathbf s}Z_{t_*}-\sum_{i=1}^ma_is_i \geq &-\frac{\beta^2\sum_{i}a_ip_i}{3\upsilon(1-{\upsilon}/{n})}\left(1-\frac{e^{\gamma }}{n^{{n}/{(2(n-\upsilon))}}}\right)\sum_{i=1}^m \frac{(s^{(i)})^2}{p_i}
    \\&-\frac{\beta^2\sum_{i}a_ip_i}{3(n-\upsilon)}\left(\frac{1}{2\upsilon} n\ln n- \frac{\gamma n}{\upsilon}\right)O({1}/{\sqrt{n}})
    \\&-\frac{\beta^2\sum_{i}a_ip_i}{3\upsilon}\left(\frac{n^{n/(2(n-\upsilon))}}{e^{\gamma }}-1\right)O({1}/{n}).
\end{align*}
The right-hand side of the above inequality converges to $-\frac{\beta^2\sum_{i}a_ip_i\sum_{i} {(s^{(i)})^2}/{p_i}}{3\upsilon}$ as $n\to \infty$ for every $\gamma>0$.

We claim that if $n$ is large enough, then there exists $ \mathbf s > \mathbf 0 $ such that \[\sum_{i=1}^ma_is_i-\frac{\beta^2\sum_{i}a_ip_i\sum_{i} {(s^{(i)})^2}/{p_i}}{3\upsilon}>0.\] 
Consider $\mathbf s =\zeta \mathbf p$ where $0<\zeta <1$ is a constant to be determined.
We want to find $\zeta$ such that \[\sum_{i=1}^ma_ip_i\zeta-\frac{\beta^2\sum_{i}a_ip_i\sum_{i} {(p_i\zeta)^2}/{p_i}}{3\upsilon}>0,\] which is equivalent to \[3\upsilon>\beta^2 \zeta.\]
From \thref{upsilon}, $\upsilon >0$, so $\frac{3\upsilon}{\beta^2} \mathbf p>\mathbf s>\mathbf 0  $ assures that the inequality in the claim holds, and such a positive magnetization $\mathbf s \in\mathcal S$ exists since $n$ is large and $0 \leq \beta<\beta_{cr}$ (if $\beta=0$, choose $\mathbf s= \mathbf p$).
 
By the last claim, for large $n$, there exists $\mathbf s \in \mathcal S $ and $\varepsilon>0$ such that \[\mathbb E_{\mathbf s}(\sum_{i=1}^m a_iS^{(i)}_{t_*}) \geq 2\varepsilon g^{t_*}\geq 2\varepsilon \frac{e^{\gamma}}{n^{n/(2(n-\upsilon))}}\geq \varepsilon\frac{e^{\gamma}}{\sqrt{n}}. \]
\thref{magvariancebound} and the Cauchy-Schwartz inequality imply $\mathbb{V}\mathrm{ar}(\sum_{i=1}^m a_iS^{(i)}_{t_*}) = O(\frac{1}{n})$ as $n\to \infty$.
Thus, by \thref{zerospin} and \thref{statistics}, for some $c>0$, \begin{align*}
    \lim _{\gamma \to \infty }\liminf_{n\to \infty}d_n(t_n- \frac{\gamma n}{\upsilon}  ) \geq \lim_{\gamma \to \infty}1-\frac{c}{\varepsilon^2 e^{2\gamma}}=1.
\end{align*}

\end{proof}

\section{Exponentially slow mixing in the low temperature regime}\label{section6}
Using a standard bottleneck ratio argument, we can show that the mixing time for the Glauber dynamics is exponential in the low temperature regime. 
The bottleneck ratio is defined as \[\Phi \colonequals \min_{A: \mu(A) \leq 1/2} \frac{\sum_{x\in A,y\notin A} \mu(x)P(x,y)}{\mu(A)}\] where $P$ is the transition matrix of the Glauber dynamics. 
The bottleneck ratio gives a lower bound of the mixing time (see \parencite[Theorem 7.4]{Peres}): \[t_{\mathrm{mix}}\geq \frac{1}{4\Phi}.\]

We need another characterization of the critical temperature $\beta_{cr}$.
\begin{lem}\thlabel{slowmixinglemma}
We have that \[\beta_{cr}= \frac{\sum_i a_i^2 p_i}{(\sum_i a_i p_i)^2 - \sum_i a_i ^2 p_i ^2}\]
\end{lem}
\begin{proof}
From $\mathbf a^T \mathbf A= g \mathbf a^T$, equation \eqref{betacrdef}, and \thref{upsilon}, we have \[\sum_i a_i p_i=\Big(p_k+\frac{1}{\beta_{cr}} \Big)a_k\] for each $k=1,\dots,m$.
Multiplying $a_k p_k$ to both sides and summing over $k$ yields the result.
\end{proof}

\begin{proof}[Proof of \thref{lowtempresult}]
It suffices to show that $\Phi \leq c_1 \exp(-c_2 n)$ for some positive constants $c_1, c_2 >0$.
By symmetry of the Hamiltonian, we have that $\mu (A) \leq 1/2$ where $A \colonequals \{\sigma : \sum_i S^{(i)}(\sigma) > 0 \}$.
Since the only way to go from $A$ to $A^c$ is to go through $B\colonequals \{\sigma : | \sum_i S^{(i)}(\sigma) | \leq 1/n \}$, it holds that \[\sum_{x\in A,y\notin A} \mu(x)P(x,y) \leq \mu(B).\]
Note that for any $\sigma \in \Omega$, \[\mu(\sigma)=\frac{\exp \bigg(\frac{\beta n}{2}\Big(\big(\sum_i S^{(i)}(\sigma)\big)^2-\sum_i \big(S^{(i)}(\sigma)\big)^2\Big)\bigg)}{Z(\beta)}.\]
By the Cauchy-Schwartz inequality,  \[\mu(B)\leq \binom{n}{\lceil n/2 \rceil} \frac{\exp \Big(\frac{\beta n}{2}\big(1-\frac{1}{m}\big) \big(\frac{1}{n}\big)^2\Big)}{Z(\beta)}\lesssim \binom{n}{\lceil n/2 \rceil}/Z(\beta)\] where $\lesssim$ denotes that the inequality holds for sufficiently large $n$ up to a constant not depending on $n$.
Using Stirling's formula, \[\Phi \lesssim \frac{\exp(n\ln 2)}{Z(\beta)\mu(A)}.\]

Now, consider the configurations with exactly $k_i n p_i$ many "$+$" spins in $J_i$ where $1/2\leq k_i\leq 1$ for each $i=1,\dots,m$ and there exists at least one $i$ such that $1/2 <k_i$.
These configurations are members of $A$ and there are at least $\prod_{i=1} ^m\binom{np_i}{k_inp_i}$ many such configurations.  
Using Stirling's formula again, we obtain \[
    Z(\beta)\mu(A) \gtrsim \Bigg(\frac{1}{\prod_{i=1}^m (1-k_i)^{p_i(1-k_i)}k_i^{p_i k_i}}\Bigg)^n e^{ \frac{\beta n}{2}\big((\sum_i (2k_i -1)p_i)^2-\sum_i  (2k_i -1)^2p_i^2\big)}
.\]
Define a function $f$ through the equation \[e^{nf(k_1,\dots,k_m)}\colonequals \Bigg(\frac{1}{\prod_{i=1}^m (1-k_i)^{p_i(1-k_i)}k_i^{p_i k_i}}\Bigg)^n e^{ \frac{\beta n}{2}\big((\sum_i (2k_i -1)p_i)^2-\sum_i  (2k_i -1)^2p_i^2\big)}. \]
Put $(k_1,\dots,k_m)= (1/2,\dots,1/2)+\gamma (v_1,\dots,v_m)$ where $v_i \geq 0$ for each $i=1,\dots,m$, $\gamma \in \mathbb R$, and $\sum_i v_i ^2 \neq 0$.
Fixing $v_i$'s, we can regard $f$ as a one-variable function of $\gamma$, say $f=f(\gamma)$, and this is equivalent to fixing a direction in $\mathbb R^m$.
A little calculation shows that \begin{align*}
    f(\gamma)&=2\beta \gamma^2 \bigg(\Big(\sum_i v_i p_i\Big)^2-\sum_i v_i^2 p_i^2 \bigg)
    \\& \quad -\sum_i p_i \big((1/2-\gamma v_i)\ln(1/2-\gamma v_i)+ (1/2+\gamma v_i)\ln(1/2 +\gamma v_i)\big)
    \\f'(\gamma)&= 4\beta \gamma \bigg(\Big(\sum_i v_i p_i\Big)^2-\sum_i v_i^2 p_i^2 \bigg)-\sum_i p_i v_i\big(-\ln(1/2-\gamma v_i)+ \ln(1/2 +\gamma v_i)\big)
    \\ f''(\gamma)&=4\beta \bigg(\Big(\sum_i v_i p_i\Big)^2-\sum_i v_i^2 p_i^2 \bigg)-\sum_ip_iv_i^2 \bigg(\frac{1}{1/2-\gamma v_i}+\frac{1}{1/2 +\gamma v_i}\bigg)
\end{align*}
where $'$ denotes a differentiation in $\gamma$.
Note that $f(0)= \ln 2$ and $f'(0)=0$.
Thus, it suffices to show that there is a direction $(v_1,\dots,v_m)$ such that $f''(0)>0$.
\thref{slowmixinglemma} shows that the direction $(v_1,\dots,v_m)=(a_1,\dots,a_m)$ satisfies $f''(0)>0$ whenever $\beta>\beta_{cr}$, which completes the proof.
\end{proof}
\begin{rmk}
Combined with the non-exponential mixing time of $O(n\ln n)$ whenever $\beta <\beta_{cr}$, the above proof shows that $\inf_{\mathbf v\geq \mathbf 0, \mathbf v\neq \mathbf 0}\frac{\sum_i v_i^2 p_i}{(\sum_i v_i p_i)^2 - \sum_i v_i ^2 p_i ^2}$ is achieved with the direction $(v_1,\dots,v_m)=\mathbf a^T$.
\end{rmk}

\begin{ack}
    The author would like to thank Professor Insuk Seo for introducing the problem and sharing his limitless insight through numerous discussions. The author also acknowledges an anonymous user at \textbf{math.stackexchange.com} \footnote{https://math.stackexchange.com/q/3553425} for the main idea of the proof in \thref{newlemma}.
    Finally, the author acknowledges the anonymous reviewers for their helpful comments and careful reading of the paper.
\end{ack}

\printbibliography

\end{document}